\documentclass{amsart}

\usepackage{eucal}

\usepackage{multicol}

\makeatletter

\def\theglossary{\@restonecoltrue\if@twocolumn\@restonecolfalse\fi
\columnseprule\z@ \columnsep 35\p@
\let\@makessectionhead\indexsec
\@xp\section\@xp*\@xp{\glossaryname}%
\let\item\@idxitem
\parindent\z@  \parskip\z@\@plus.3\p@\relax
\footnotesize}

\def\glossaryname{Notation Index}

\makeatother

\makeglossary

\numberwithin{equation}{section}
\renewcommand\dots{\relax\ifmmode\ldots\else$\,\ldots\,$\fi}

\newcommand\note[1]%
{$^\dagger$\marginpar{\sf\footnotesize{$^\dagger${#1}}}}

\divide\count253 by 60 
\divide\count255 by 60 
\multiply\count255 by 60 
\advance\count254 by -\count255 

\def\today{\number\day\space\ifcase\month\or January\or February\or
March\or April\or May\or June\or July\or August\or September\or
October\or November\or December\fi\space\number\year}

\def\hour{\ifnum\count253<10
0\number\count253\else\number\count253\fi}

\def\minute{\ifnum\count254<10
0\number\count254\else\number\count254\fi}

\swapnumbers
\newtheorem{theorem}[equation]{Theorem}
\newtheorem{proposition}[equation]{Proposition}

\newtheorem{corollary}[equation]{Corollary}

\theoremstyle{definition}

\newcommand\gk{\mathfrak{k}} 
\newcommand\gt{\mathfrak{t}}
\renewcommand\gg{\mathfrak{g}}
\newcommand\gh{\mathfrak{h}} 
\newcommand\gn{\mathfrak{n}} 
\newcommand\gu{\mathfrak{u}} 
\newcommand\gs{\mathfrak{s}}
\newcommand\go{\mathfrak{o}} 
\newcommand\gp{\mathfrak{p}} 
\newcommand\gq{\mathfrak{q}} 

\newcommand\gb{\mathfrak{b}}
\newcommand\gl{\mathfrak{l}} 

\newcommand\bb[1]{{\text{\bf#1}}}

\newcommand\bbz{\bb{Z}} 

\newcommand\bbr{\bb{R}} 
\newcommand\bbc{\bb{C}}

\newcommand\bbp{\bb{P}}
\newcommand\bbi{\bb{I}}

\newcommand\bbt{\bb{T}}

\newcommand{\cP}{\mathcal P}
\newcommand{\CC}{{\mathbb C}}
\newcommand{\RR}{{\mathbb R}}
\newcommand{\ZZ}{{\mathbb Z}}

\renewcommand{\a}{\alpha}
\renewcommand{\b}{\beta}

\newcommand\ca{\mathcal}

\newcommand\func[1]{\operatorname{\mathrm{#1}}}
\newcommand\funclim[1]{\operatorname*{\mathrm{#1}}}

\newcommand\ad{\func{ad}}

\renewcommand\dim{\func{dim}}

\renewcommand\lim{\funclim{lim}}

\renewcommand\max{\func{max}}

\newcommand\tr{\func{tr}}

\renewcommand{\a}{\alpha}
\newcommand{\cR}{\mathcal R}

\DeclareMathOperator{\aff}{aff}

\begin{document} 
\title{Loop groups and holomorphic bundles}
\author{Jacques Hurtubise}
\address{Department of Mathematics and Statistics, McGill University\\
           805 Sherbrooke St. W, Montreal, Quebec, H3A 2K6, Canada\\}
\email{jacques.hurtubise@mcgill.ca}
\author{Michael Murray}
\address{School of Mathematics and Statistics, University of Adelaide\\
           Adelaide, Australia\\}
\email{michael.murray@adelaide.edu.au}
\thanks{The first author acknowledges the support of the Natural Sciences and Engineering Research Council of Canada. The second author acknowledges the support of the Australian  
Research Council and thanks Mike Eastwood for useful conversations.}

\begin{abstract} This paper considers the links between the geometry of the various flag manifolds of loop groups and bundles over families of rational curves. Aa an application, a stability result for the moduli on a rational ruled surface of $G$-bundles    with additional flag structure along a line  is proven for any  reductive group; this gives the corresponding stability statement for any compact group $K$ for the moduli of  $K$-instantons over the four-sphere, and for the moduli of $K$-calorons over the three-sphere times the circle. The paper also considers the Hecke transforms on bundles induced by outer automorphisms of the loop groups.

\end{abstract}
\maketitle

\tableofcontents

\section{Introduction}

A beautiful correspondence brought to light by Atiyah \cite{Atiyah} ties together the geometry of bundles on families of rational curves, parametrised by a variety $V$,  to maps of $V$ into a certain flag manifold of a loop group. The correspondence is quite simple: one associates to a bundle a suitable transition function which is defined over a circle. This correspondence has been exploited repeatedly, for example in describing moduli of instantons: \cite{Kirwan-AJ,BHMM2,Gr,HM-ruled,Tian1,Tian2}.

This article revisits this correspondence, with an eye to the full range of flag manifolds of the loop group. Many of the features brought to light in the more restricted case considered above also hold here: for example, when the variety $V$ is a curve, the spaces of maps admit a ``poles and principal parts" description. 
This allows a proof of topological stability, that is, that the homotopy groups stabilise as one takes the degrees of the maps to infinity. We treat here the case of $V= \bbp^1$. 

While these spaces of maps are of interest in their own right, they are also moduli spaces: for a suitably chosen maximal parabolic, one has moduli spaces of instantons on $S^4$, and this correspondence was the one exploited in \cite{Kirwan-AJ,BHMM2,Tian1,Tian2} to give the proof of the Atiyah-Jones conjecture on topological stability. The more general parabolics correspond to moduli of calorons (instantons with monopole-like asymptotics on $\bbr^3\times S^1$). Below, we get topological stability theorems for all instanton moduli space on $S^4$, previous work having just considered classical groups, as well as for the more general caloron moduli.

Another aspect that we consider is the role that outer automorphisms play in this picture; we find that these automorphisms correspond to Hecke transforms. They will define isomorphisms, some of them not obvious, between various instanton, caloron and monopole moduli.

\section{Parabolic subgroups of loop groups}

Let $K$ be a compact semi-simple Lie group with complexification $G$. Choose a maximal torus $T \subset K$.  Denote the Lie algebras of $K$, $G$ and $T$ by 
$\gk$, $\gg$ and $\gt$. Let $\gh = \gt^c$ which is a Cartan subalgebra of $\gg$. Denote by $\cR(\gg) \subset \gh^* $ the set of roots and choose a  base of simple roots  $\triangle(\gg)$. As usual define  a root to be positive if it is a positive linear combination of simple roots.   If $\a$ is a root denote  the corresponding root space by $\gg_\a$. 
The Borel subalgebra $\gb$ is defined by 
$$
\gb =  \gh \oplus_{\a >  0} \gg_\a
$$
and the  opposite Borel is defined by 
$$
\gb^- = \gh  \oplus_{\a < 0} \gg_\a.
$$
Similarly define
$$
\gn =  \oplus_{\a > 0} \gg_\a
$$
and its opposite  by 
$$
\gn^- =  \oplus_{\a <  0} \gg_\a.
$$
Note that we have 
$$
\gg = \gn^- \oplus   \gh   \oplus \gn.
$$
 If $\triangle_0 \subset \triangle(\gg)$ we define a parabolic subalgebra $\gp(\triangle_0)$ as the Lie subalgebra of $\gg$ generated by 
$$
\gb \oplus \bigoplus_{\a \in \triangle_0} \gg_{-\a}.
$$
This  in fact establishes a bijective correspondence between subalgebras containing $\gb$ and subsets of $\triangle(\gg)$. Notice that
$\gp(\emptyset) = \gb$ and $\gp(\triangle(\gg)) = \gg$. 

To extend this to loop algebras it is useful to take a different, but equivalent approach. Let $\cR_\ZZ(\gg)\subset \gh^*$ be the $\ZZ$-module generated
by the roots. In particular if $\triangle(\gg)$ is a base for $\cR(\gg)$ then  $\cR_{\ZZ}(\gg)$ is a free $\ZZ$-module over 
$\triangle(\gg)$. 
Define a function  $\chi \colon \cR_\ZZ(\gg) \to \ZZ$ by assigning a number which is one or zero to each positive simple root and 
extending it to be $\ZZ$-linear.  For any $\a \in \cR_\ZZ({\gg})$ let us define $\gg_\a$ to be the corresponding weight space. So 
if $\a = 0$ then $\gg_0 = \gh$, if $\a \in \cR(\gg)$ then $\gg_\a$ is the corresponding root space and otherwise $\gg_\a$ is zero. 
 Noting that $[\gg_\alpha, \gg_\beta]\subset \gg_{\alpha+\beta}$ we see that 
$$
\bigoplus_{\chi(\a) \geq 0}  \gg_\a 
$$
is the  parabolic subalgebra corresponding to 
$$
\triangle_\chi = \{ \a \in \triangle(\gg) \mid \chi(\a) = 0 \}.
$$

Consider the loop algebra 
$$
L\gk = \RR[z, z^{-1}] \otimes_\RR \gk
$$
with the pointwise bracket as well as its complexification $L\gg = \CC[z, z^{-1}] \otimes_\CC\gg$. 
It is  convenient to consider the semi-direct product $\tilde L \gg = L\gg \oplus  \CC d$ whose elements we 
write as $X(z) + x d$. This has  bracket:
$$
[X(z) + x i d, Y(z) +  yid] = [X(z), Y(z)] + x Y'(z) - y X'(z)
$$
where $X'(z)$ is the derivative with respect to $z$. 
We define $\delta \colon \gh \oplus \CC d \to \CC$ by $\delta(X(z) + x d) = x$ and extend any linear function on  $\gh$ to 
$\gh \oplus \CC d$ by making it  zero  on $d$. \  The root space decomposition is
$$
\tilde L \gg  = (\oplus_{(n, \a) \in \ZZ \times \cR(\gg)}  z^n \otimes \gg_\a )\oplus \gh \oplus id\CC.
$$
and the roots of $\tilde L \gg$ are
$$
\cR(\tilde L \gg) = \ZZ \times \cR(\gg) = \{ n\delta + \a \mid \a \in \cR(\gg), n \in \ZZ \} \subset (\gh \oplus \CC d)^*
$$
 Notice that we can write  
 $$
\tilde L\gg = \bigoplus_{n < 0} z^n \gg \oplus \gn^- \oplus \gh \oplus id\CC \oplus \gn \oplus \bigoplus_{n > 0} z^n \gg .
$$
We say that a root $n \delta + \a$ is positive if $n > 0$ or $n=0$ and $\a > 0$. 

If $\triangle(\gg) = \{ \a_1, \dots, \a_r\}$ consider the positive roots
$\a_0 = \delta - \theta, \a_1, \a_2, \dots, \a_r$; here $\theta$ is the highest root of $\gg$. These are simple roots for this definition of positive.  To see this note that
$$
n\delta + \a = n(\delta - \theta) + (\theta + \a) + (n-1)\theta.
$$
If $n> 0$ then this is a non-negative linear combination of the simple roots. If $n = 0$ then $\a > 0$ so it is also a non-negative combination 
of simple roots.  So we choose the base of simple roots
$$
\triangle(\tilde L \gg) = \{ \a_0 = \delta - \theta, \a_1, \a_2, \dots, \a_r \}
$$

The Borel subalgebra formed by taking non-negative root spaces is 
$$
 d\CC \oplus \gb \oplus \bigoplus_{n > 0} z^n \gg 
$$
which has the geometric interpretation as loops $X(z)$ such that $X(0) \in \gb$.

To understand the parabolic subalgebras note that we have  $\triangle(\tilde L \gg) = \{ \delta - \theta\} \cup  \triangle(\gg) $
as  simple roots of the loop algebra. Let $\chi \colon \cR_\ZZ(\tilde L \gg) \to \ZZ$ be a function  defined by associating  one or   zero to each simple root and
extending to be $\ZZ$-linear.  Then as before we can consider 
$$
  \bigoplus_{\chi(\b) \geq  0}  {\tilde L \gg}_{\b} = \gh \oplus \CC d \oplus \bigoplus_{\chi(n\delta + \a) \geq 0} z^n \gg_\a 
$$
which is the parabolic subalgebra defined by 
$$
\triangle_\chi = \{ \b \in \triangle({\tilde L \gg}) \mid \chi(\b)= 0 \}
$$

We need to consider when it is possible for $\chi(n \delta + \a) = n \chi(\delta - \theta) + n \chi(\theta) + \chi(\a) \geq 0$. Notice that
because $\theta$ is the longest root we have $\chi(\theta) \geq \chi(\alpha)$ for all roots $\alpha \in \cR(\gg)$. 

Consider the case that $n \leq -2$ then
\begin{align*}
\chi(n \delta + \a) &\leq -2 \chi(\delta - \theta) - 2 \chi(\theta) + \chi(\a) \\
                    & = -2 \chi(\delta - \theta) - \chi(\theta) + (\chi(\a) - \chi(\theta))
                    \end{align*}
We have $-2\chi(\delta - \theta) \leq  0$ and $\chi(\a) - \chi(\theta) \leq 0$ so the only way this expression can be non-negative
is if $\chi(\theta) = 0$ and hence $\chi(\a) = 0$ for every root $\a$. Thus, if $\chi$ is non zero, $\chi(\delta - \theta) =1$ and so the 
expression is still negative. We conclude that the expression must be negative for $n \leq -2$. 

Next consider $n = -1$ so
\begin{align*}
\chi(n \delta + \a) &=  -\chi(\delta - \theta) +  (\chi(\a) - \chi(\theta)).  
                \end{align*}
It is only possible for this to be non-negative if $\chi(\delta - \theta) = 0$ and $\chi(\a) = \chi(\theta)$.  Define
$$
\gq_\chi = \bigoplus_{\chi(\a) = \chi(\delta)} \gg_{\a}.
$$
Notice that if $\chi(\delta - \theta) = 1 $ then $\gq_\chi = 0$. 

If $n=0$ we have
$$
\chi(n \delta + \a)  = \chi(\a)
$$
and thus we are interested in 
$$
\bigoplus_{\chi(\a) \geq 0} \gg_\a.
$$
If we denote $\chi$ restricted to $\cR_\ZZ(\gg)$ also by $\chi$ then we have a parabolic subalgebra
$$
\gp( \triangle_\chi(\gg)) = \bigoplus_{\chi(\a) \geq 0} \gg_\a.
$$

Finally if $n \geq 1$ then we have 
$$
\chi(n \delta + \a) \geq \chi(\delta + \a) =  \chi(\delta - \theta) + \chi(\theta + \a) \geq 0.
$$

We conclude that $\chi$ determines the parabolic subalgebra
$$
\gp(\triangle_\chi(\tilde L \gg)) = \frac{1}{z}  \gq_\chi  \oplus \gp( \triangle_\chi(\gg))\oplus \gh \oplus \CC d \oplus \bigoplus_{n > 0} z^n \gg.
$$

Notice that there are two distinct cases:

\begin{enumerate}
\item If $\chi(\delta - \theta) = 1$ then this parabolic subalgebra is 
$$
\gp( \triangle_\chi(\gg)) \oplus \gh \oplus \CC d \oplus \bigoplus_{n > 0} z^n \gg.
$$
which has the geometric interpretation as  loops $X(z)$ such that $X(0) \in \gp( \triangle_\chi(\gg))$. We call such parabolics {\em standard} and we call 
$\gp( \triangle_\chi(\gg)) $ the related finite parabolic.
\item  If $\chi(\delta - \theta) = 0$ then $\gq_\chi  \neq 0$. We call such parabolics {\em exotic}.
\end{enumerate}

To understand the relationship between these two types of parabolic subalgebras it is useful to consider the (extended) Dynkin diagram of the 
loop algebra. This is the diagram constructed from the Dynkin diagram of the finite algebra by adding the simple root $\alpha_0 = \delta - \theta$ according
to the usual rules for forming a Dynkin diagram.  Parabolics can be read off from a  Dynkin diagram by using 
the notation of \cite{BasEas}. Each parabolic is determined  by placing a cross through every  node $\alpha_i$ of the 
Dynkin diagram for which $\chi(\alpha_i) =  1$. So for example the Borel subalgebra  has a cross through every node and a maximal parabolic subalgebra has a cross through  only one node. A parabolic is standard if the extended Dynkin diagram has a cross through the node $\alpha_0 = \delta - \theta$ and   the related finite parabolic is the parabolic determined by the corresponding finite Dynkin diagram.  A parabolic is exotic if the node $\alpha_0$ 
does not have a cross through it. 

As in the finite case automorphisms of the extended Dynkin diagram induce outer automorphisms of the loop algebra. These can be chosen so that they fix the Cartan subalgebra and act on the simple roots according to the original automorphism.  They then act on the parabolic subalgebras by permuting the crossings
of the nodes.  Inspection of the extended Dynkin diagrams below shows that the action can sometimes interchange standard and exotic parabolics.

\subsection{An example}

Let $G = SL(2, \CC)$ and likewise $\gg = sl(2, \CC)$.  Conjugation by the matrix
$$
X = \begin{bmatrix} 0 & z^{-1/2} \\   z^{1/2} & 0 \end{bmatrix}
$$
is an outer automorphism of the loop algebra (and group) (called a {\em flip} in \cite{Nye}).  Notice that it acts as
$$
\begin{bmatrix} a & b \\ c & d \end{bmatrix} \mapsto \frac{1}{z} \begin{bmatrix} 0 & c \\ 0 & 0  \end{bmatrix} + \begin{bmatrix} d & 0 \\ 0 & a \end{bmatrix} + 
                                        z  \begin{bmatrix} 0 & 0  \\ b & 0 \end{bmatrix}.
                                        $$
                                        
Let  $\gb_{LSL(2, \CC)}$ be the standard Borel subalgebra and group of the loop algebra and group defined by 
$$
\gb_{LSL(2, \CC)} = \left\{ \begin{bmatrix} a & b \\ 0 & c \end{bmatrix}  + O(z) \right\}
$$
where $O(z^k)$ denotes an arbitary power series in $z$ with lowest order $z^k$ and coefficients in $sl(2, \CC)$. 
If we let 
$$
\begin{bmatrix} a & b \\ 0 & c \end{bmatrix}  + z \begin{bmatrix} d & e  \\ f & g \end{bmatrix} + O(z^2) 
$$ 
be an arbitrary element of $\gb_{LSL(2, \CC)}$ and apply the outer automorphism it goes to 
$$
\begin{bmatrix} c & 0 \\ 0 & a  \end{bmatrix} + z \begin{bmatrix} 0 & 0  \\ b & 0 \end{bmatrix} + 
\begin{bmatrix} 0 &  f \\ 0  & 0  \end{bmatrix} + z \begin{bmatrix} g & 0  \\ 0 & d \end{bmatrix} + z\begin{bmatrix} 0 & * \\ 0 & 0 \end{bmatrix} + 
O(z^2)
$$
or 
$$
\begin{bmatrix} c & f  \\ 0 & a  \end{bmatrix} + O(z).
$$
In other words $B_{LSL(2, \CC)}$ is fixed by the outer automorphism.  

Consider the parabolic 
$$
\cP_1 = \{ g(z) \in LSL(2, \CC) \mid g(0) \ \text{exists}\ \}
$$
which contains $\gb_{LSL(2, \CC)}$. 
If we apply the outer automorphism we must get another parabolic containing $\gb_{LSL(2, \CC)}$. Let an arbitary element in $\cP_1$ be
$$
\begin{bmatrix} a & b \\ c & d \end{bmatrix}  + z \begin{bmatrix} e & f \\ g & h \end{bmatrix} +O(z^2) 
$$ 
and as above apply the outer automorphism to get an element of the form
$$
\frac{1}{z} \begin{bmatrix} 0 & c  \\ 0 & 0 \end{bmatrix}  + \begin{bmatrix}d & g \\ 0 & a \end{bmatrix}+O(z) .
$$
This is an arbitrary element in the parabolic
$$
\cP_2 = \left\{ \frac{1}{z} \begin{bmatrix} 0 & c \\ 0 & 0 \end{bmatrix}  + \begin{bmatrix}d & g \\ 0 & a \end{bmatrix} + O(z) \right\}
$$
described in the previous section.

\subsection{An alternative way to approach the example}

A more symmetric  way of understanding what is going on is to represent the loop algebra or group in terms of infinite dimensional matrices. For convenience we consider polynomial loops. 
Let $L\CC^2 = \CC[z^{-1}, z] \otimes \CC^2$. Define an infinite basis of this by
$$
e_{2i} = \begin{bmatrix}1 \\ 0 \end{bmatrix} z^i  \qquad\mathrm{and}\qquad e_{2i+1} = \begin{bmatrix}0 \\ 1 \end{bmatrix}z^i
$$
for $i \in \ZZ$. An element of the loop algebra of the form 
$$
M_j = \begin{bmatrix} a & b \\ c & d \\ \end{bmatrix} z^j
$$
acts by 
\begin{align*}
e_{2i} & \mapsto a e_{2i+2j} +    b e_{2i+2j+1} \\
e_{2i+1} & \mapsto c e_{2i+2l} + d e_{2i+2j+1} 
\end{align*}
and hence by the infinite  matrix $M_{ik}$ whose only non-zero entries are
$$
\left[ \begin{array}{cccc}
M_{2i, 2i+2j} &  M_{2i, 2i+2j+1 } \\
M_{2i+1, 2i+2j} &  M_{2i+1, 2i+2j+1} \\
\end{array} \right]
= 
\left[ \begin{array}{cccc}
a & b  \\
c & d  \\
\end{array} \right]
$$
for all $i \in \ZZ$.  This matrix shares with the general element of the loop algebra and indeed the group  the property that 
$M_{i,k} = M_{i+2, k+2}$. We  have a non-trivial outer automorphism of $L\gs\gl(2,\bbc)$ which corresponds to 
$M_{i, k} \mapsto M_{i+1, k+1}$. This is the flip described in the previous section.

The Borel subalgebra defined above maps to the upper triangular  matrices. The standard parabolic
and the exotic parabolic map to the block upper triangular matrices, periodic with two by two blocks. Flipping interchanges the standard and the exotic   parabolic and corresponds to moving everying along the 
diagonal by one entry. 

More generally consider the case of polynomial loops in the general linear algebra $\gg\gl(n,\bbc)$.  As above
we can think of the matrix representation of such  a loop as an infinite dimensional matrix $M= (M_{k,l})_{k, l \in \ZZ}$
indexed by $k, l \in \ZZ$. For such a matrix $M$ define $M_{[i,j]}$ to be the $n \times n$ matrix 
which is the $(i, j)$th block
$$
M_{[ij]} = \left[\begin{array}{cccc}M_{ni+1, nj+1 } & M_{ni+1, nj+2 } & \cdots & M_{ni+1, nj + n } \\
M_{ni+2, nj+1} & M_{ni+2, nj+2 } & \cdots & M_{ni+2, nj+n } \\
\vdots & \vdots & \ddots & \vdots \\
M_{ni+n, nj+1} & M_{ni+n, nj+2 } & \cdots & M_{ni+n, nj+n }\end{array}\right].
$$
As in the $2 \times 2$ case the loop algebra element $\sum_{j} M_j z^j$ is represented  by the infinite matrix $M$ with 
$M_{[i,j]} = M_{i-j}$ for all $i, j \in \ZZ$.  These matrices $M$ are obviously invariant
under $M_{i, j} \mapsto M_{i+n, j+n}$.  Notice that if we take loops in a Lie subalgebra $L \subset \gg\gl(n, \CC)$ then 
for every $i, j \in \ZZ$ we have $M_{[i, j]} \in L$.

For the particular   classical algebras and groups, one has:

\subsection{The special linear algebra $\gs\gl(n,\bbc)$}.

\begin{figure}
\setlength{\unitlength}{.75 mm}
\begin{picture}(100, 30)
\put(10,10){\circle{2}}
\put(8,5){$\alpha_1$}
\put(11,10){\line(1,0){13}}
\put(25,10){\circle{2}}
\put(23,5){$\alpha_2$}
\put(26,10){\line(1,0){13}}
\put(42,10){\circle*{1}}
\put(45,10){\circle*{1}}
\put(48,10){\circle*{1}}
\put(51,10){\line(1,0){13}}
\put(65,10){\circle{2}}
\put(63,5){$\alpha_{\ell -1}$}
\put(66,10){\line(1,0){13}}
\put(80,10){\circle{2}}
\put(78,5){$\alpha_{\ell}$}
\put(11,10.5){\line(3, 1){33}}
\put(79,10.5){\line(-3, 1){33}}
\put(45,21.5){\circle{2}}
\put(43, 25){$\alpha_0$}
\end{picture}
\caption{The Dynkin diagram of the loop algebra of $\gs\gl(n,\bbc)$}
\label{fig:ddal1}
\end{figure}

From the infinite dimensional matrix point of view, the loop algebra consists of matrices periodic under the shift of indices $(i,j)\mapsto (i+n, j+n)$; simple root spaces for the $i$-th root correspond to above diagonal matrix entries $M_{kn+i, kn+i+1}$; the opposite root spaces to the entries below the diagonal $M_{kn+i+1, kn+i}$.
The parabolics, both standard and exotic, correspond to block upper triangular matrices, where again the pattern of blocks along the diagonal is $n$-periodic; the standard parabolics correspond to the case for which the entry $(1,0)$ does not belong to the blocks.

There is an outer automorphism, this time of order $n$, given by conjugation by:

\begin{equation}\label{outer-sl}
A(z) = \begin{bmatrix} 0&0&0&\dots&0&z^{\frac{-n+1}{n}}\\
z^{\frac{1}{n}}&0&0&\dots&0&0\\
0&z^{\frac{1}{n}}&0&\dots&0&0\\
\dots&\dots&\dots&&\dots&\dots\\
0&0&0&\dots&z^{\frac{1}{n}}&0
\end{bmatrix}.
\end{equation}
It permutes the simple roots for $\gs\gl(n,\CC)$ and the extra root $\delta-\theta$ cyclically. 
In the infinite matrix picture, the automorphism corresponds to the shift of indices $(i,j)\mapsto(i+1,j+1)$, and so permutes the parabolic subalgebras. In particular, any parabolic can be mapped to one of the standard ones. This  automorphism rotates  the Dynkin diagram in Figure \ref{fig:ddal1}. 

Notice that there is a reflection automorphism of the Dynkin diagram in Figure \ref{fig:ddal1} which is not a composition of 
rotations. This leaves $\a_0$ fixed and swaps $\a_1 \mapsto \a_l$, $\a_2 \mapsto \a_{l-1}$, \dots and is induced by the 
corresponding automorphism of the Dynkin diagram of $\gs\gl(n,\CC)$.

\subsection{Orthogonal algebras in even dimension}

Let us define our pairing on $\CC^{2n}$ by $\langle e,f\rangle = \sum_{i=1}^{2n} e_if_{2n+1-i}$. Matrices in the Lie algebra of the orthogonal group then satisfy the symmetry relation:
\begin{equation}
M_{i,j} = - M_{2n+1-j, 2n+1-i},\label{orthosymmetry}
\end{equation}
 and the Borel subalgebra is the subalgebra of upper triangular matrices. Elements of the loop algebra, represented as infinite matrices, then have an invariance under an infinite dihedral group, acting along the diagonal, generated by the extension of the symmetry in the matrices for $\gs\go(2n,\bbc)$ and the translational symmetry
$M_{i,j} = M_{i+2n,j+2n}$. The parabolic subalgebras are the given by subspaces of block upper triangular matrices, where the pattern of blocks is $2n$-periodic along the diagonal, and addition is invariant under the reflection \ref{orthosymmetry}.

There is an outer automorphism, given by conjugation by 
\begin{equation}
\begin{bmatrix}
0&z^{\frac{1}{2}}\bbi\\ z^{\frac{-1}{2}}\bbi&0
\end{bmatrix}.\label {slip}
\end{equation}
Its effect on the infinite matrix is the translation $(i,j)\mapsto (i+n,j+n)$. This acts on the  Dynkin diagram in Figure \ref{fig:dddl1} by sending
the simple root $\alpha_i$ to $\alpha_{n-i}$. 

There is in addition another outer automorphism, inherited from the finite dimensional algebra, where it acts by conjugation by 
\begin{equation}
\label{eq:perm}
\left[\begin{array}{ccccc|ccccc}
1&0&\dots &0&0&0&0&\dots &0&0\\
0&1&\dots &0&0&0&0&\dots &0&0\\
\vdots &\vdots &\ddots &\vdots&\vdots&\vdots&\vdots&\ddots &\vdots&\vdots\\
0&0&\dots &1&0&0&0&\dots &0&0\\
0&0&\dots &0&0&1&0&\dots &0&0\\
\hline
0&0&\dots &0&1&0&0&\dots &0&0\\
0&0&\dots &0&0&0&1&\dots &0&0\\
\vdots&\vdots&\ddots &\vdots&\vdots&\vdots&\vdots&\ddots &\vdots&\vdots\\
0&0&\dots &0&0&0&0&\dots &1&0\\
0&0&\dots &0&0&0&0&\dots &0&1\\
\end{array}\right]\end{equation}
where each of the blocks is $n \times n$. 
\begin{figure}
\setlength{\unitlength}{.75 mm}
\begin{picture}(100,25)
\put(25,15){\circle{2}}
\put(23,10){$\alpha_2$}
\put(26,15){\line(1,0){13}}
\put(42,15){\circle*{1}}
\put(45,15){\circle*{1}}
\put(48,15){\circle*{1}}
\put(51,15){\line(1,0){13}}
\put(65,15){\circle{2}}
\put(63,10){$\alpha_{\ell - 2}$}
\put(66,14.5){\line(2,-1){13}}
\put(80,8){\circle{2}}
\put(83,21){$\alpha_{n-1}$}
\put(66,15.5){\line(2,1){13}}
\put(80,22){\circle{2}}
\put(83,7){$\alpha_n$}
\put(24,14.5){\line(-2,-1){13}}
\put(10, 22){\circle{2}}
\put(2, 22){$\alpha_{0}$}
\put(24,15.5){\line(-2,1){13}}
\put(10, 8){\circle{2}}
\put(2,6){$\alpha_1$}
\end{picture}
\caption{The Dynkin diagram of the loop algebra of $\gs\go(2n, \CC)$}
\label{fig:dddl1}
\end{figure}

On the infinite matrix, it permutes the entries in certain pairs of adjacent rows, switching, for all $k,i$, $M_{(2k+1)n, j}$ and $M_{(2k+1)n +1, j}$, then permutes the entries in corresponding pairs of adjacent columns, switching, for all $k,i$ $M_{i,(2k+1)n }$ and $M_{i, (2k+1)n +1},$. This is of course represented by conjugation by an appropriate infinite periodic permutation matrix. This fixes all simple roots except $\alpha_{n-1}$ and $\alpha_n$ which it 
swaps.   Combining this with the previous automorphism we can construct an automorphism which leaves the extended Dynkin diagram 
fixed except for swapping $\a_0$ and $\a_1$. This corresponds to conjugation by the infinite periodic
permuation matrix arising from \eqref{eq:perm} shifted by $n$. Together these generate the whole automorphism group which is $D_4$.

Notice that the automorphisms fix the subset of simple roots $\{\a_0, \a_1, \a_{n-1}, \a_n \}$ and act transitively on it. It follows that the only  parabolics $\gp(\triangle_\chi)$ which can be mapped by an automorphism to a standard parabolic (i.e $\a_0 \notin \triangle_\chi$)
are precisely those for which  $\triangle_\chi$ does not contain the full set $\{\a_0, \a_1, \a_{n-1}, \a_n \}$.

\subsection{Orthogonal algebras in odd dimension}
Let us again define our pairing on $\CC^{2n+1}$ by $\langle e,f\rangle = \sum_{i=1}^{2n} e_if_{2n+2-i}$. Matrices in the Lie algebra of the orthogonal group then again satisfy the symmetry relation
$M_{i,j} = - M_{2n+2-j, 2n+2-i}$, and the Borel subalgebra is the subalgebra of upper triangular matrices.

Going to the loop algebra, and again representing the elements as infinite matrices, we have that the matrices are in a similar way to even dimensions invariant under the action of an infinite dihedral group acting along the diagonal. Again, the parabolics are given by block upper triangular matrices, where the pattern of blocks is invariant under the action of the dihedral group.

\begin{figure}
\setlength{\unitlength}{.75 mm}
\begin{picture}(100,30)
\put(25,10){\circle{2}}
\put(23,5){$\alpha_2$}
\put(26,10){\line(1,0){13}}
\put(42,10){\circle*{1}}
\put(45,10){\circle*{1}}
\put(48,10){\circle*{1}}
\put(51,10){\line(1,0){13}}
\put(65,10){\circle{2}}
\put(63,5){$\alpha_{\ell - 1}$}
\put(66,10.5){\line(1,0){13}}
\put(66,9.5){\line(1,0){13}}
\put(72,9){$\rangle$}
\put(80,10){\circle{2}}
\put(79,5){$\alpha_n$}
\put(24,9.5){\line(-2,-1){13}}
\put(10, 17){\circle{2}}
\put(2,17){$\alpha_{0}$}
\put(24,10.5){\line(-2,1){13}}
\put(10, 3){\circle{2}}
\put(2,1){$\alpha_1$}
\end{picture}
\caption{The Dynkin diagram of the loop algebra of $\gs\go(2n+1, \CC)$}
\label{fig:ddbl1}
\end{figure}

As we can see from the extended Dynkin diagram in Figure \ref{fig:ddbl1}
there is  only one outer automorphism, there is no translation by a half period along the diagonal, as the period is of odd length. The outer automorphism is given by permuting the entries $M_{ k(2n+1), j}$ and $M_{k(2n+1) +1, j}$, then permuting the entries $M_{i,k(2n+1) }$ and $M_{i, k(2n+1) +1}$. Again, it is given by conjugation by a suitable (infinite, periodic) permutation matrix.

\subsection{Symplectic algebras} Let us define our pairing on $\CC^{2n}$ by 
$$\langle e,f\rangle = \sum_{i=1}^{n} e_if_{2n+1-i}-f_ie_{2n+1-i}.$$
 We decompose matrices $M$ in the Lie algebra of the symplectic group in $n\times n$ blocks as 
\begin{equation}
\begin{bmatrix}
M^1&M^2\\M^3&M^4
\end{bmatrix}.\end{equation}
The blocks then satisfy the symmetry relations
$M^1_{i,j} = - M^4_{n+1-j, n+1-i}, M^2_{i,j} = M^2_{n+1-j, n+1-i}, M^3_{i,j} =  M^3_{n+1-j, n+1-i}.$
Translating into infinite  matrix terms for the loop algebras, we get for the loop algebra of the symplectic group matrices invariant again under an infinite dihedral group acting along the diagonal, but this time with different changes of signs associated to the reflections.

Again considering the Dynkin diagram in Figure \ref{fig:ddcl1} we see that there is  only one outer automorphism, associated  to the translation $(i,j)\mapsto (i+n,j+n)$. 

\begin{figure}
\setlength{\unitlength}{.75 mm}
\begin{picture}(100,15)
\put(25,10){\circle{2}}
\put(23,5){$\alpha_1$}
\put(26,10){\line(1,0){13}}
\put(42,10){\circle*{1}}
\put(45,10){\circle*{1}}
\put(48,10){\circle*{1}}
\put(51,10){\line(1,0){13}}
\put(65,10){\circle{2}}
\put(63,5){$\alpha_{\ell - 1}$}
\put(66,10.5){\line(1,0){13}}
\put(66,9.5){\line(1,0){13}}
\put(72,9){$\langle$}
\put(80,10){\circle{2}}
\put(79,5){$\alpha_n$}
\put(11,10.5){\line(1,0){13}}
\put(11,9.5){\line(1,0){13}}
\put(15,9){$\rangle$}
\put(10,10){\circle{2}}
\put(5,5){$\alpha_0$}
\end{picture}
\caption{The Dynkin diagram of the loop algebra $\gs\gp(2n, \CC)$}
\label{fig:ddcl1}
\end{figure}

\section{Flag manifolds of a loop group, and link to bundles}

We thus have  parabolic subalgebras $\gp_{LG}$, and hence subgroups $P_{LG}$, associated to each loop group $LG$. We now turn to the corresponding homogeneous spaces.

Following Atiyah, we associate to any element of $LG$ a $G$-bundle over the Riemann sphere, simply by thinking of the loop as a transition function defined along the equator over the Riemann sphere. Of course, bundles are in bijective correspondence with the double cosets
$LG_-\backslash LG/LG_+$, where $LG_+$, $LG_-$ are the subgroups of loops extending holomorphically to $|z|\leq 1, |z|\geq 1$, respectively. $LG_+$ is one of the maximal parabolic subgroups discussed earlier.

\begin{proposition}\cite{Atiyah}
Elements of $LG/LG_+$ are in one to one correspondence with pairs $(E,t)$, where $E$ is a $G$-bundle on $\bbp^1$ and $t$ is a trivialisation of $E$ on the disk $|z|\ge 1$.

More generally, if $P_{LG}$ is a standard parabolic subgroup corresponding to a parabolic $P_G$ of $G$, elements of $LG/P_{LG}$ are in one to one correspondence with triples $(E,t, R)$, where $E$ is a $G$-bundle on $\bbp^1$,  $t$ is a trivialisation of $E$ on the disk $|z|\ge 1$, and $R$ is a reduction to $P_G$  over $z=0$.
\end{proposition}

The case of the ``exotic" parabolics is a little bit harder to see. Fortunately, there are the outer automorphisms that allow us to move them to standard parabolics, essentially by Hecke transforms of the associated vector bundles, as we shall see below.

One can of course consider not only individual points, but families, in particular those parametrised by a Riemann surface $\Sigma$. 

\begin{proposition}\cite{Atiyah}
Elements of $Hol(\Sigma, LG/LG_+)$ are in one to one correspondence with pairs $(E,t)$, where $E$ is a $G$-bundle on $\Sigma\times \bbp^1$ and $t$ is a trivialisation of $E$ on the product of $\Sigma$ and the disk $|z|\ge 1$. One can add in a basing condition, so that for  $Hol^*(\Sigma, LG/LG_+)$ one asks over the basepoint of $\Sigma$ that the  trivialisation of the bundle extend to the whole $\bbp^1$ lying over the basepoint.

More generally, if $P_{LG}$ is a standard parabolic subgroup corresponding to a parabolic $P_G$ of $G$, elements of $Hol(\Sigma,LG/P_{LG})$ are in one to one correspondence with triples $(E,t, R)$, where $E$ is a $G$-bundle on $\Sigma\times \bbp^1$,  $t$ is a trivialisation of $E$ on the product of $\Sigma$ and the disk $|z|\ge 1$, and $R$ is a reduction to $P_G$  over $z=0$.
\end{proposition}

In particular, If $\Sigma$ is $\bbp^1$, a trivialisation of a bundle on $\bbp^1\times \{z=\infty\} \cup \{x=\infty\}\times \bbp^1$ extends automatically to some neighbourhood of $z=\infty$, and, using the $\bbc^*$ action, one can pull back to a bundle trivialised over $|z|\ge 1$. Thus:

\begin{proposition}\cite{Atiyah}
The space $Hol^*(\bbp^1, LG/LG_+)$ of based maps is homotopy equivalent to the space of pairs $(E,t)$, where $E$ is a $G$-bundle on $\bbp^1\times \bbp^1$ and $t$ is a trivialisation of $E$ on $\bbp^1\times \{z=\infty\} \cup \{x=\infty\}\times \bbp^1$. Over compact families, this equivalence is realised by a biholomorphic map.

More generally, if $P_{LG}$ is a standard parabolic subgroup corresponding to a parabolic $P_G$ of $G$, the space $Hol^*(\bbp^1, LG/P_{LG})$ of based maps is homotopy equivalent to the space of triples $(E,t, R)$, where $E$ is a $G$-bundle on $\Sigma\times \bbp^1$,  $t$ is a trivialisation of $E$ on $\bbp^1\times \{z=\infty\} \cup \{x=\infty\}\times \bbp^1$, and $R$ is a reduction to $P_G$  over $z=0$.
\end{proposition}

Let us look at some of the homogeneous spaces in turn:

\subsection{$Sl(n,\bbc)$ or $Gl(n,\bbc)$} Let us choose the standard normalisation, for which the root spaces correspond to the entries in the $n$ by $n$ matrices.  We start with an element of  $LGl(n,\bbc)/B_{LGl(n,\bbc)}$, the full flag space, and see what happens when we project to $LGl(n,\bbc)/P_{LGl(n,\bbc)}$. As noted, an element of $LGl(n,\bbc)/B_{LGl(n,\bbc)}$ gives us a rank $n$ vector bundle $E$ over the Riemann sphere,  with degree, let us say, zero, along with a framing near infinity , and a flag in the fiber $E_0$ over the origin:   $E_0^1\subset E_0^2\subset \dots  E_0^n= E_0$, where the superscript indicates the dimension. This flag reduces the structure group to $B_{Gl(n,\bbc)}$ over the origin, and reduces the group of compatible bases near the origin to $B_{LGl(n,\bbc)}$.

The essential point here is that such a bundle with flag corresponds not just to one bundle, but to a whole sequence. Indeed, let us define the  sheaves of sections of $E$, holomorphic away from the origin, meromorphic near the origin:
\begin{equation}\label{sl(n)-sheaves}
E^{i,j}= \{ s| {\rm order}(s) = -i,  {\rm leading\ term\ }\in E_0^j\}.
\end{equation}
One has $E= E^{0,n}$. These sheaves are torsion free and thus correspond to bundles. As a bundle, the degree of $E^{i,j}$ is $n(i-1)+j$. One has the nested sequence
 \begin{equation}
\dots \subset E^{-1,1}\subset E^{-1,2}\subset \dots  E^{-1,n}\subset E^{0,1}\subset E^{0,2}\subset \dots E^{0,n}\subset E^{1,1}\dots \label{inclusions}
\end{equation}
Near the origin, $z^kE^{i,j} = E^{i-k, j}$. Each of these bundles comes equipped by a flag at the origin, defined in essence by the $n-1$ sheaves which precede it in the sequence. This way of looking at things puts the bundle $E$ and all its Hecke transforms on the same footing.  With this, projection to $LGl(n,\bbc)/P_{LGl(n,\bbc)}$ for one of the standard parabolics amounts to deleting a certain subset $E^{0,j_s}, j_s\in \Delta_0$ of $E^{0,1}, E^{0,2} \dots  E^{0,n-1}$ from the sequence of inclusions, as well as all their ``translates" $E^{i, j_s}$ in the first index. Considering the ``exotic" parabolics also simply allows one to delete the $E^{i,n}$ as well. Thus, since parabolics correspond to a subset $\tilde\Delta_0$ of $\{1,2,\dots ,n\}$
\begin{proposition}
Let $P_{LGl(n,\bbc)}$ be the parabolic associated to $\tilde\Delta_0$. Elements of $LGl(n,\bbc)/P_{LGl(n,\bbc)}$ correspond to a infinite sequence 
 \begin{equation}
\dots \subset E^{-1,j_1}\subset E^{-1,j_2}\subset \dots  E^{-1,j_k}\subset E^{0,j_1}\subset E^{0,j_2}\subset \dots E^{0,j_k}\subset E^{1,j_1}\dots 
\end{equation}
of rank $n$ torsion free sheaves over $\bbp^1$, with $\{j_1<j_2<\dots <j_k\} = \{1,2,\dots ,n\}\backslash \tilde\Delta_0$, equipped with compatible trivialisations over $|z|\geq 1$, and such that the quotient of two successive sheaves $E^{i,j_{\ell+1} }/E^{i,j_\ell}$ is a skyscraper sheaf over $z=0$ isomorphic to ${\ca O}_0^{j_{\ell+1}-j_\ell}$, and similarly the quotient $E^{i+1,j_{1}}/E^{i,j_k}$ is a skyscraper sheaf over $z=0$isomorphic to ${\ca O}_0^{n + j_1 - j_{k}}$.

Choosing one of these sheaves as our base bundle, say, $E^{i,j_s}$, this data corresponds to a bundle of degree $ni+j_s$, trivialised over $|z|\geq 1$, and equipped with a flag of subspaces of dimensions $j_{s+1}-j_{s}, j_{s+2}-j_s,\dots j_k-j_s, n+j_1-j_s,\dots , n+j_{s-1}-j_s$.
\end{proposition} 

For example, elements of  $LGl(2,\bbc)/B_{LGl(2,\bbc)}$ one has a degree zero rank 2 bundle $E$, with a distinguished line over the origin; projecting to the two parabolics gives us in turn 1)  the bundle $E$, or 2) the bundle $\tilde E = E^{0,1}$, of degree $-1$, the Hecke transform of $E$. 

In general, the Hecke transform corresponds to the action of the outer automorphism of $LGl(n,\bbc)$ given by conjugation by the matrix of \ref{outer-sl}.

We note in ending that the case of $Sl(n,\bbc)$ is essentially similar, apart from the fact that the bundles $E^{i,j}$ are equipped with volume forms which   have zeroes or poles of order $-(n-1)i+j$ at the origin. The volume forms are compatible with each other in the obvious sense. 

\subsection{Orthogonal groups in even dimension} We start with the even rank case, and again start with the full flags. An element of $LSO(2n,\bbc)/B_{LSO(2n,\bbc)}$ corresponds to a rank $2n$ vector bundle $E$ over $\bbp_1$, equipped with a symmetric non degenerate quadratic form; it is trivialised on the disk $|z|\ge 1$, and at the origin, is equipped with an isotropic flag $E_0^1\subset E_0^2\subset \dots \subset E_0^{n-1}\subset E_0^{n,+}$, (superscripts denote dimension) determining a coisotropic flag $E_0^{n+1}=(E_0^{n-1})^\perp\subset\dots E_0^{2n-1}=(E_0^1)^\perp$ and a second isotropic $n$-plane $E_0^{n,-}$ with 
$E_0^{n-1}\subset E_0^{n,-}\subset E_0^{n+1}$. In an adapted basis $e_i$ at the origin, in which the quadratic form is given by $\sum_{i=1}^na_ia_{2n+1-i}$, the subspace $E_0^k$ for $k\ne n$ corresponds to the first $k$ elements of the basis, with $E_0^{n,+}$ corresponding to the first $n$ elements, and $E_0^{n,-}$  to the first $n+1$ elements, less the $n$th element.
In short, we have the flag:
$$
\begin{matrix}
&&&&E_0^{n,+}\cr
E_0^1\subset& E_0^2\subset& \dots & E_0^{n-1}\subset&&\subset E_0^{n+1}&\subset\dots &\subset E_0^{2n-1}\cr
&&&&E_0^{n,-}\cr
\end{matrix}.
$$

Again one wants to think of $E$ as an infinite sequence of bundles, with isotropic-coisotropic flags. As above, define the  sheaves of sections of $E$, holomorphic away from the origin, meromorphic near the origin:
\begin{equation}
E^{i,j}= \{ s| {\rm order}(s) = -i,  {\rm leading\ term\ }\in E_0^j\}
\end{equation}
Here we allow the index $j$ to also stand for $(n,+)$ or $(n,-)$. There are also  sheaves 
$ E^{i,2n,+ }$ and $E^{i,2n,-}$.  There is then a sequence of inclusions:
\begin{equation}
\label{orthosequence}
\begin{matrix}
&&&&&&&E^{i-1,n,+}\cr
\dots \subset &E^{i-1,1}&\subset& E^{i-1,2}&\subset \dots & E^{i-1,n-1}&\subset& &\subset\cr
&&&&&&&E^{i-1,n,-}&\cr\cr
&&&&&&&E^{i-1,2n,+}\cr
\dots & E^{i-1,n+1}&\subset& E^{i-1,n+2}&\subset\dots &E^{i-1,2n-1}&\subset&&\subset \cr
\cr
&&&&&&&E^{i ,2n,-}\cr\cr
&&&&&&&E^{i ,n,+}\cr
\dots &E^{i ,1}&\subset& E^{i-1,2}&\subset \dots & E^{i ,n-1}&\subset&&\subset\dots \cr
&&&&&&&E^{i ,n,-}\cr
\cr
\end{matrix}
\end{equation}
The sheaves $E^{i,k}, i\in \bbz$ all have an orthogonal structure away from the origin, which degenerates at the origin, acquiring a zero or a pole. Near the origin, the sheaf $E^{i,j}$ thought of as a bundle, has a basis   $z^{-i}e_1, .., z^{-i}e_j, z^{-i+1}e_{j+1},\dots z^{-i+1}e_{2n}$, and so, a quadratic form 
\begin{align} {\rm for}\ j\leq n:& \sum_{i=1}^j z^{-2i+1} a_ia_{2n+1-i} +  \sum_{i=j+1}^n z^{-2i+2} a_ia_{2n+1-i}\\
{\rm for}\ n<j\leq 2n,& \sum_{i=1}^{j-n} z^{-2i} a_ia_{2n+1-i} + \sum_{i=j-n+1}^n z^{-2i+1} a_ia_{2n+1-i}.\label{quadform}
\end{align}

For the sheaves $E^{i,n},E^{i,2n}, i\in \bbz$, alternately, one can think of the fibers as being equipped with a conformal structure, which extends in a non degenerate way to the origin. The sheaves again have a privileged isotropic-coisotropic flag at the origin, with the passage from $E^{i,n}$ to $E^{i,2n}$ exchanging isotropic and coisotropic. Again, this last Hecke transform can be though of as being mediated by the outer automorphism of the loop group defined by conjugating by \ref{slip}. In terms of infinite matrices, this automorphism moves us either up or down by $n$ steps along the diagonal. 

The orthogonal groups in $2n$ dimensions already have an outer automorphism, which in terms of flags permutes the half-dimensional isotropic subspaces $E_0^{n,+}$  and $E_0^{n,-}$, and leaves the other spaces of the flag in place; this of course extends to the loop group. When one goes to the loop group, one acquires a similar automorphism, permuting $E_{i,2n,+}$ and $E_{i,2n,-}$, but leaving the other subspaces in place. These two automorphisms, together with \ref{slip}, generate a group of order 4, corresponding to the symmetries of the extended Dynkin diagram.

When one passes from the Borel to the various standard parabolics, one deletes as for the case of $Sl(n,\bbc)$ a certain number of  subspaces $E_0^{j_l}$, including possibly $E_0^{n,\pm}$, of the isotropic flag at the origin, as well as their coisotropic twins $(E_0^j)^\perp = E_0^{2n-j_l}$. In terms of the infinite sequence of sheaves, there is a corresponding infinite sequence of deletions; the exotic parabolics amount to including $E^{i,2n,\pm}$ in the possible exclusions. It is important to note that there is some redundancy in the sequence \ref{orthosequence}; for example, $E_0^{n,+}$ and $E_0^{n,-}$ together determine $E_0^{n-1}$, $E_0^{n+1}$. Thus deletions might not actually determine a new flag manifold; however subsets $\tilde\Delta_0$ of the simple roots do correspond to periodic deletions $D_0$ in the sequence above.

\begin{proposition}
Let $P_{LSO(2n,\bbc)}$ be the parabolic associated to $\tilde\Delta_0$. Elements of $LSO(2n,\bbc)/P_{LSO(2n,\bbc)}$ correspond to a infinite sequence of rank $2n$ torsion free sheaves, with inclusions modelled on \ref{orthosequence} with periodic deletions $D_0$. Successive quotients are the same as for $Gl(2n,\bbc)$. The  sheaves $E^{i,j}$ remaining in the list have compatible quadratic forms \ref{quadform};
 they are also equipped with compatible orthonormal trivialisations over $|z| >  1$.

There is one Hecke transformation which acts by permuting $E_0^{n,\pm}$, another by permuting $E_0^{2n,\pm}$, and a final one which acts by shifting the second index by $n$.
\end{proposition}

\subsection{Orthogonal groups in odd dimension} 
The odd dimensional case is in some aspects similar, but in others distinct: for the group $SO(2n+1,\bbc)$, one obtains a rank $2n+1$ bundle, trivialised as before on the disk $|z|\ge 1$, and at the origin, equipped with an isotropic flag $E_0^1\subset E_0^2\subset \dots \subset E_0^{n}$, which extends to the orthogonal complement coisotropic flag $E_0^{n+1}=(E_0^{n})^\perp\subset\dots E_0^{2n}=(E_0^1)^\perp$, giving 
an infinite chain
\begin{equation}\label{orthosequence2}
E^{-1,2n+1}\subset E^{0,1}\subset E^{0,2}\subset \dots \subset E ^{0,2n+1}\subset E^{1,1}\dots 
\end{equation}
 Again, passing to the parabolics gives us deletions. 
On the other hand, there is only one non-trivial outer automorphism, this time permuting $E^{i,2n+1}$ and $E^{i+1, 1}$.

\begin{proposition}
Let $P_{LSO(2n+1,\bbc)}$ be the parabolic associated to $\tilde\Delta_0$. Elements of $LSO(2n+1,\bbc)/P_{LSO(2n+1,\bbc)}$ correspond to a infinite sequence of rank $2n+1$ torsion free sheaves, with inclusions modelled on \ref{orthosequence2} with periodic deletions $D_0$. Successive quotients are the same as for $Gl(2n,\bbc)$. The  sheaves $E^{i,j}$ remaining in the list have compatible quadratic forms as in the even case;
 they are also equipped with compatible orthonormal trivialisations over $|z| >  1$.
\end{proposition}

\subsection{Symplectic groups} This case is essentially similar to that of the orthogonal groups in even dimensions, but with an isotropic-coisotropic flag.
$$
E_0^1\subset  E_0^2\subset  \dots   E_0^{n-1}\subset E_0^n\subset E_0^{n+1} \subset\dots  \subset E_0^{2n-1}
$$
giving the infinite sequence of sheaves:
\begin{equation}\label{symplsequence}
E^{-1,2n }\subset E^{0,1}\subset E^{0,2}\subset \dots \subset E ^{0,2n }\subset E^{1,1}\dots 
\end{equation}
\begin{proposition}
Let $P_{LSp(n,\bbc)}$ be the parabolic associated to $\tilde\Delta_0$. Elements of $LSp(n,\bbc)/P_{LSp(n,\bbc)}$ correspond to a infinite sequence of rank $2n$ torsion free sheaves, with inclusions modelled on \ref{symplsequence} with periodic deletions $D_0$. Successive quotients are the same as for $Gl(2n,\bbc)$. The  sheaves $E^{i,j}$ remaining in the list have compatible quadratic forms \ref{quadform};
 they are also equipped with compatible orthonormal trivialisations over $|z| >  1$.

There is one Hecke transformation which acts  by shifting indices by $(i,j)\mapsto (i,j+n) (j\leq n)$ or $(i,j)\mapsto (i+1,j-n) (j>n)$.
\end{proposition}

\section{Birkhoff, Bruhat, Borel-Weil}

\noindent {\it i)  The Birkhoff and Bruhat stratifications.}
 Just as for finite dimensional groups, we have an action of the unipotent radical $N^-_{LG}$ opposite to $B_{LG}$
on $LG/ P_{ LG}$; this action sweeps out contractible sets $\Sigma_w$ of finite codimension, the Birkhoff strata. The unipotent radical 
$N_{LG}$ of $B_{LG}$, in turn, sweeps out cells $C_w$ of finite dimension. The two sets $\Sigma_w$ and $C_w$ are transverse, and meet at a point.

The cells are indexed by elements $w$ of the Weyl group of $\widehat{LG}$, also known as the affine Weyl group of $G$. It is a semi-direct product $W_{\aff} =W\times \check T$, where $W$ is the Weyl group of $G$, and $\check T$ the lattice of homomorphisms $\bbc^*\rightarrow T$, where $T$ is a maximal torus of $G$. It is generated by the simple root reflections 
$\sigma_i$ for $i=0, 1, \dots, r$.  Again just as in the finite case we can define the length of an element $w \in  W_{\aff}$, $\ell(w)$, as the minimal number of simple root reflections in a  minimal presentation of $w$. 

For ${P_{LG}}$ a parabolic subgroup corresponding to a set of simple roots $\triangle_0$ we can define the subgroup $W_{P_{LG}} \subset W_{\aff}$ generated by the 
root reflections for simple roots in $\triangle_0$. This is the Weyl group of the semi-simple Levi factor of ${P_{LG}}$.  Inspection of the extended Dynkin diagrams show that this is always a finite dimensional group and hence 
$W_{P_{LG}}$ is finite.  There is a unique element in each right 
coset of $W_{P_{LG}} $ in $W_{\aff}$ of minimal length. This defines a subset $W^{P_{LG}}$ called the Hasse diagram of ${P_{LG}}$.  Just as in case of finite-dimensional semisimple Lie algebras we have that every $w \in W_{\aff}$ can be uniquely decomposed as $w = w_1 w_2$ with $w_1 \in W^{P_{LG}}$ and 
$w_2 \in W_{P_{LG}}$ and $\ell(w) =  \ell(w_1) + \ell(w_2)$.  In the finite case these results are proved in for example \cite{Hum}
and the extension to the infinite, or affine, case also follows from the discussions therein.  If ${P_{LG}}$ is not exotic and corresponds to a finite parabolic $Q$ then $W^{P_{LG}} = W^Q \times \check T$.  We need one final result which can 
be proved using the results in \cite{Hum}. That is that if $w \in W^{P_{LG}}$ then we have
$$
\{ \b > 0 \mid w(\b) < 0  \} \subset \cR(u),
$$
where $\cR(u)$ are the roots of the root spaces occuring in the Lie algebra of $u$ of the unipotent 
subalgebra of $P_{LG}$.

Since $W_{\aff}$ is the quotient of the normaliser  $N(\bbc^*\times T)$ of the torus $\bbc^*\times T$ of $\tilde LG$  by $\bbc^*\times T$, it embeds well into the quotient $\tilde LG/P_{LG}$. Our contractible open sets will then be orbits of the elements $w$ of $W_{\aff}^P$.

Before stating the next theorem recall that we have:
\begin{align*}
P_{LG} &= Levi_{LG} U_{LG}, \\
B_{LG} &= T^c N_{LG}, \\
\end{align*}
as well as an embedding (cf \cite{ PS}) dense decomposition of 
$$N^-_{LG} B_{LG} \subset LG$$
as a dense open set.

\begin{theorem}(cf \cite{ PS}, p.144. for the case when $P$ is a Borel.)
 \label{birkhoff-strata}

(i) The complex manifold $LG/P_{LG}$ is the union of strata $\Sigma_w$ indexed by elements $w$ of $W^{P_{LG}}$.

(ii) The stratum $\Sigma_w$ is the orbit of $w$ under $N^-_{LG}$, and is a locally closed contractible complex submanifold of $LG/P_{LG}$ of codimension
$\ell(w)$. It is diffeomorphic to $U^w_{LG} = N^-_{LG}\cap wU^-_{LG}w^{-1}$.

(iii) In particular $\Sigma_1$ is an open dense subset of $G/P_{LG}$ which is acted on freely and transitively 
by $U_{LG}^-$. 

(iv) Let $A_w = N_{LG}\cap w U^-_{LG}w^{-1}$. The orbit of $w$ under $A_w$ is a complex cell $C_w$ of dimension $\ell(w)$ which intersects $\Sigma_w$ transversally at $w$.

\end{theorem}
\begin{proof}
The proof proceeds as in the finite-dimensional case and the case where $P_{LG}$ is a Borel
as discussed in \cite{PS}. In particular we note from \cite{PS} the factorisation theorem:
\begin{equation}\label{bruhat-borel}
LG = \bigcup_{w \in W_{\aff}} N^-_{LG} w B_{LG}. 
\end{equation}

(i) There is a projection map 
$$LG/B_{LG}\rightarrow LG/P_{LG}$$ 
whose fibre is the finite dimensional full flag manifold of the Levi factor.  The decomposition \ref{bruhat-borel} projects to a stratification of $LG/P_{LG}$, and as the elements $w, w'$ of $W_{\aff}$ give the same stratum if and if only $w'w^{-1}$ lies in $W_{P_{LG}}$, the images are classified by the elements of the Hasse diagram $W^{P_{LG}}$. They are also disjoint. Note that the orbit corresponding to $w=1$ is dense.

(ii) As the isotropy subgroup of the ${LG}$ action at
$wP_{LG}$ is $wP_{LG}w^{-1}$ it follows that the action of $U^w_{LG} = N^-_{LG}\cap wU^-_{LG}w^{-1}$ is free, so that  $\Sigma_w $
is diffeomorphic to $U^w_{LG}$. (iii) then follows for $w=1$. To get the codimension result, we note that the orbits in $LG/B_{LG}$ projecting to $N^-_{LG}w P_{LG}$ correspond to the elements in the coset $wW_{P_{LG}}$; of these, one is the dense orbit corresponding to $w\in W^{P_{LG}}$ (as $w$ is of minimal length), and the codimension for this orbit below is the same as its codimension above, that is $\ell(w)$, by the result on \cite{PS}. 

(iv) The intersection of $C_w$ and $\Sigma_w$ will be the orbit of $wP_{LG}$ under the group 
$$
A_w \cap U^w_{LG} = N_{LG}\cap w U^-_{LG}w^{-1} \cap N^-_{LG}\cap wU^-_{LG}w^{-1} = \{ 1 \}.
$$
A similar calculation with Lie algebras shows that the tangent spaces to $C_w$ and $\Sigma_w$ at $wP_{LG}$ have zero intersection so the  intersection is transversal, and so its dimension is the codimension of $\Sigma_w$. 

To calculate the dimension of $C_w$, in an alternate way,  it is sufficient to calculate the dimension of $N_{LG}\cap w U^-_{LG}w^{-1}$
or the dimension of $\gn \cap w \gu^- w^{-1}$.  As this can be decomposed into root spaces we just need the cardinality of 
$$ 
\{ \a > 0 \mid w^{-1}(\a)  \in \cR(\gu^-) \}.
$$
If we map $\a$ to $\b = -w^{-1}(\a)$ we see that this is the same 
as the cardinality of 
$$
\{ \b > 0 \mid w(\b) < 0 \quad\text{and}\quad \b \in \cR(\gu)\} . 
$$
But we have already noted that for $w \in W^{P_{LG}}$ we have
$$
\{ \b > 0 \mid w(\b) < 0  \} \subset \cR(u)
$$
so that 
$$
\{ \b > 0 \mid w(\b) < 0 \quad\text{and}\quad \b \in \cR(\gu) \}
= \{ \b > 0 \mid w(\b) < 0  \}
$$
and thus has cardinality $\ell(w)$. 


\end{proof}

We remark that in the case of $P_{LG} = LG^+$, the cell $C_w, w\in (\check T)^+$ is a versal deformation space for deforming the $G$-bundles on $\bbp^1$ given by $w$; here $(\check T)^+$ is the positive Weyl chamber in $\check T$.

\bigskip

\noindent{\it ii) The Borel-Weil theorem.}
The classical Borel-Weil theorem has an analog for loop groups \cite{PS}. Let $\lambda\in \Lambda_{LP}$ be a fundamental lowest weight, with $\bbc_\lambda$ the corresponding weight space. It is a representation of the torus $\bbt\times T$ in $\tilde LG$, and extends trivially to
$ P_{\tilde LG}$.  One can build the homogeneous line bundle over $\tilde LG/P_{\tilde LG}$
\begin{equation}
L_\lambda = \tilde LG\times_{P_{\tilde LG}}\bbc_\lambda.
\end{equation}
The space of sections of $L_\lambda$ is a representation $V_\lambda$ of $\tilde LG$, with lowest weight space $\bbc_\lambda$. These basic lowest weight sections are  defined by
\begin{align}
\tilde LG/P_{\tilde LG}\times \bbc_\lambda &\rightarrow \tilde LG\times_{P_{\tilde LG}}\bbc_\lambda\cr
(g, e)&\mapsto (g, \pi(g^{-1}\cdot e)).
\end{align}
Here $\pi:V_\lambda\rightarrow\bbc_\lambda$ is the projection.
 The basic sections in $\bbc_\lambda$ vanish on a divisor $Z_\lambda$. We choose a basis element $v_\lambda\in \bbc_\lambda$. 
We set  $n(P)+1$ to be the cardinality of $\Lambda_{LP}$.

The group $U_{LG}$ preserves the divisors $Z_\lambda$; indeed, $Z_\lambda$ is the union of the codimension one Birkhoff stratum $\Sigma_{w(\lambda)}$, and of other, higher codimensions cells $\Sigma_{w'}$ contained in the closure of $\Sigma_{w(\lambda)}$. Here $w(\lambda)$ is the reflection in the simple root plane corresponding to $\lambda$. One has a decomposition of $LG/P_{LG}$ as the union of the big cell which is  a free orbit of $N_{LG}$, and the divisors $Z_\lambda$. We set 
\begin{equation}
Z= \cup_{\lambda\in \Lambda_P}Z_\lambda.
\end{equation}

From this, one has (\cite{PS}):

\begin{proposition}
The divisors $Z_\lambda$, as $\lambda$ varies in $\Lambda_{LP}$, span freely the second homology group $H^2(\tilde LG/P_{\tilde LG}, \bbz)\simeq \bbz^{n(P)+1}$.
\end{proposition}

\begin{proposition}
Based maps $F$ from a Riemann surface $\Sigma$ into $LG/P_{LG}$ are classified up to homotopy by a multi-degree ${\bf k} = (k_1,..,k_{n(P)})$ of non-negative numbers, given by the intersection multiplicity of $F(\Sigma)$ with the $Z_{\lambda}$, or alternately, by the first Chern classes of the pullbacks $F^*(L_\lambda)$.
\end{proposition}

Let $P_{LG}$ be a parabolic subgroup coresponding to a parabolic of $G$. Then $n(P)-1$ of the degrees of maps to $LG/P_{LG}$ correspond to degrees of maps into $G/P$, in the following fashion. For each  $w_i$  generator of the Hasse diagram of $G/P$, then there is a corresponding divisor $Z_{w_i}$, the closure of  $Bw_iP$ in $G/P$, such that the degree $(j_1,...,j_{n(P)-1})$ of a map of a Riemann surface into counts the intersection number of the image with the divisors $(Z_{w_1},...,Z_{w_{n(P)-1}})$. Corresponding to these, the degrees $(k_1,..,k_{n(P)-1})$ of a based map of the Riemann surface into can be defined as follows: 
 a based map from $\Sigma$ into  $LG/P_{LG}$  corresponds to a holomorphic bundle on $\Sigma\times \bbp^1$, trivialised on $\Sigma\times \{|z|>1\} \cup \{x=\infty\}\times \bbp^1$, and with a   $P$-flag   $F^0$ along $z=0$. The degree $k_i$ counts with multiplicity the number of lines $\{x\}\times \bbp^1$ 
over which there are sections of the associated $G/P$-bundle taking  $F^0$ as value at $z=0$ and lying in $Z_{w_i}$ at $z=\infty$. The final degree $k_{n(P)}$ is simply the degree (second Chern class) of the bundle; It counts, again with suitable multiplicities, the number of lines in the ruling over which the bundle is non-trivial.

Thus, for example, taking $G= Sl(n, \bbc)$ and $P_{LG} = B_{Sl(n,\bbc)}$, a holomorphic map corresponds to a holomorphic vector bundle $E$ on $\Sigma\times \bbp^1$ with trivial volume form, equipped with a flag $E_0^1\subset E_0^2\subset\cdots\subset E_0^n=E_0$ over $z=0$, and, because it is trivialised at $z=\infty$ a (transverse) flag $E_\infty^1\subset E_\infty^2\subset\cdots\subset E_\infty^n=E_\infty$ over $z=\infty$. The degree $k_i$ counts the number of lines in the ruling such that $E$ has a section over that line lying in $E_0^i$ at $z=0$ and $E_\infty^{n-i}$ at $z=\infty$. More properly, one had not only a bundle $E$, but a sequence $E^{i,j}$ ( \ref{sl(n)-sheaves}); one can consider the subsheaf $\hat E^{0,j}$ of $E^{0,j}$ consisting of sections taking value at $z=\infty$ in the subspace $E_\infty^{n-j}$. Let $\pi$ be the projection $\pi:\Sigma\times \bbp^1 \rightarrow \bbp^1$; the direct image sheaf $R^1\pi_*(\hat E^{0,j})$ is supported over a discrete set. Its degree (dimension over $\bbc$) is the $j$th degree $k_j$. Similarly, let $\hat E$ be the subsheaf of $E= E^{0,n}$ of sections vanishing over $z=\infty$; the  degree $k_n$ (second Chern class) of the bundle is the complex dimension of $R^1\pi_*(\hat E)$.

More concretely, if one has an $Sl(n,\bbc)$ vector bundle over $\Sigma\times \bbp^1$ with $c_2 = k$, and a flag $0= E_0\subset E_1\subset E_2\subset...\subset E_{n-1}\subset E_n = E$ over $z=0$, with $c_1(E_i/E_{i-1}) = -j_i$, then the degrees defined above are given by:
\begin{eqnarray}\label{charges}
k_i&=& k+j_1+...+j_i,\quad i=1,..,n-1,\\ k_n &=&  k.
\end{eqnarray}

\section{Principal parts and deformation theory}

\noindent{\it i) The sheaf of principal parts.} The article \cite{BHMM1}, following on previous work \cite{Segal, Kirwan-maps, Gu}, studies the spaces of maps from the Riemann sphere into the finite dimensional flag manifolds; the main conceptual tool for  \cite{BHMM1}  is an idea of  \cite{Gr}, that of the principal part of a holomorphic map into a flag manifold. Using this, it was proven in \cite {BHMM1} that the homotopy type of the space of holomorphic  maps into the flag manifold stabilises to that of the continuous maps as one increases the degree. The same idea can be used to extend the result  from maps to flag manifolds to the more general case of maps to  equivariant compactifications of a solvable Lie group in \cite {BHM}; this general case covers, for example, toric varieties, as well as certain spherical varieties.

We are now interested in maps from a Riemann surface $\Sigma$ into one of the flag manifolds of a loop group. This infinite dimensional case was in fact the prime subject of Gravesen's original article. Just as for finite dimensional groups, we can consider the flag manifold $LG/P_{LG}$ as ``adding an infinity" to the contractible group $N_{LG}$. While we cannot speak of compactification, it is very much as if there was a compactification when one is studying rational maps. 

Just as in \cite{Gr,HuM2, BHMM1} we define a ``sheaf of principal parts" over a Riemann surface $\Sigma$ of maps into $LG/P_{LG}$. We proceed by analogy with maps into   $\bbp^1$, the compactification of $\bbc$. These can be thought of as meromorphic maps into $\bbc$; the principal part of the map at a pole is obtained by quotienting the map by the additive action of holomorphic maps into $\bbc$.  In the same vein, one thinks of maps into $LG/P_{LG}$ as meromorphic maps into  $N_{LG}$; the principal part is obtained by quotienting by the action of holomorphic maps into $N_{LG}$. One thus defines a presheaf, by setting, for each open set $V$,
\begin{equation}
{\mathcal P}r (V) = Hol (V, LG/P_{LG})/Hol(V, N_{LG}).
\end{equation}
This map equips the space of principal parts on $V$ with the quotient topology. Note that because the action is free and transitive over the big open cell in  $LG/P_{LG}$, there is only one principal part (the trivial part) associated to maps with values in this cell.  Sections of the sheaf of principal parts resemble sections of the sheaf of divisors, in that sections are trivial over a Zariski open set. As for divisors, this allows us to add  two sections with disjoint support.
Choosing a base point in the big cell allows us to think of the sheaf of principal parts as being defined by an exact sequence  (of pointed sets)
\begin{equation}\label{ppt-sheaf}
0\rightarrow Hol(V, N_{LG})\rightarrow Hol (V, LG/P_{LG}) \rightarrow {\mathcal P}r (V)\rightarrow 0.\end{equation}

Since the action of $N_{LG}\subset U_{LG}$ fixes the holomorphic sections $v_\lambda$ of the homogeneous line bundles $L_\lambda$, it preserves the divisors $Z_\lambda$, and so each principal part has a well defined (multi)-degree ${\mathbf m}$ in $\bbz^{n(P)}$.

\bigskip
\noindent{\it ii) Rational maps and the sheaf of principal parts.} The spaces $LG/P_{LG}$ admit a ``scaling" equivalence onto the homogeneous space $LG^a/P_{LG}^a$ of loops on the circle $|z|=|a|$, by  associating to the loop $f$ the loop $f^a(z) = f(z/a)$. More generally, there are equivalences $LG^a/P_{LG}^a\rightarrow LG^b/P_{LG}^b $.  However, using the identifications 
\begin{equation}
LG^a/P_{LG}^a={\rm bundles\ (plus\ flags)\ on\ }\bbp^1\ {\rm trivialised\ over\ }|z|\ge |a|,
\end{equation}
there is another map   given by restriction for $|a|\le|b|$:
\begin{equation}
r_{a,b}: LG^a/P_{LG}^a\rightarrow LG^b/P_{LG}^b.
\end{equation}
Let us take a limit, and set 
\begin{align}
LG^\infty/P_{LG}^\infty=&\{ {\rm \ bundles\ (plus\ flags)\ on\ }\bbp^1\ \cr
&\quad {\rm trivialised\ over\ some\ disk\ centred\ at\ } z=\infty\}.
\end{align}

\begin{proposition}
A based map $F$ of $\bbp^1$ into $LG^\infty/P_{LG}^\infty$, with basing condition $F(\infty) = \bbi$, is equivalent to a holomorphic $G$-bundle over $\bbp^1\times\bbp^1$, trivialised over $\{\infty\}\times\bbp^1\cup\bbp^1\times\{\infty\}$, with a reduction to $P$ over $z=0$.
\end{proposition}
\begin{proof} It is clear that such a map $F$ is equivalent to a holomorphic $G$-bundle over $\bbp^1\times\bbp^1$, trivialised over $V\times\bbp^1\cup\bbp^1\times\{\infty\}$, for some open set $V$ containing infinity.  However, if a bundle is trivial over $\infty\times\bbp^1$, it is trivial over $V\times\bbp^1$ for some open set $V$; a trivialisation on $\{\infty\}\times\bbp^1\cup\bbp^1\times\{\infty\}$ extends to $V\times\bbp^1\cup\bbp^1\times\{\infty\}$.
\end{proof}

\begin{corollary}
A compact family, parametrised by $K$, of holomorphic $G$-bundles over $\bbp^1\times\bbp^1$, trivialised over $\{\infty\}\times\bbp^1\cup\bbp^1\times\{\infty\}$, is equivalent to a compact family, parametrised by $K$, of based maps $F$ of $\bbp^1$ into $LG^a/P_{LG}^a$,  for some $a$, with basing condition $F(\infty) = \bbi$. 
\end{corollary}

\begin{theorem}\label{configurations}
The map which associates to a holomorphic map of degree $\mathbf k$ in $Hol^*(\bbp^1,LG^\infty/P_{LG}^\infty)$ of degree $\mathbf k$ the corresponding section in 
$H^0_{\mathbf k}(\bbc, {\mathcal P}r)$ of the space of  configurations of principal parts of total degree $\mathbf k$ is an isomorphism.
\end{theorem}
\begin{proof} The proof follows that given in \cite {BHMM2}.
We begin by noting that open sets $V$ mapped into the big cell, where the principal part is trivial, yield trivial $G$-bundles over $\bbp^1\times V$   equipped with a reduction to $B_G$ at $\{\infty\}\times V$ and a reduction to $N_G$ over  $\{0\}\times V$. The fact that the bundles are trivial on $\bbp^1\times V$ means that the reduction at $0$ can be transported to $\infty$; as we are on the big cell, we have that these two reductions, to $N_G$ and $B_G$, intersect transversally in a point in the fiber of $E$ over $\infty$. In other words, over the big cell, the bundles we obtain are naturally trivialised at $\infty$, and indeed then trivialised globally on $\bbp^1\times V$. Two such bundles can then be glued in a unique way.

Returning to our map, we see that it is naturally injective; a configuration of principal parts determines transition functions, and the equivalence relation in their definition translates into bundle equivalence. Surjectivity follows from a glueing construction. One covers $\bbp_1$ by disjoint disks $V_i, i= 1,\dots ,j$ surrounding the points $p_i$ which are mapped outside the big cell, and by $V_0=\bbp^1\backslash\{p_1,p_2,\dots ,p_j\}$; one can lift the principal part over the $V_i$ to $\phi_i\in Hol(V_i,LG/P_{LG})$, and so bundles $E_i$ which we then glue, using the unique glueing, to the trivial bundle over $\bbp^1\times V_0$ over the intersection $\bbp^1\times (V_0\cap V_i)$.
\end{proof}

Following the model of maps of $\bbp^1$ into $\bbp^1$, we refer to the points of the source $\bbp^1$ over which there is a non-trivial principal part as the {\it poles} of the map. As noted above, each pole has a well defined multiplicity ${\mathbf m} = (m_0,m_1,..,m_{n(P)})$, given by the order of vanishing of the pullbacks $v_{\lambda_i}$ at the pole. 

There are well defined pole location maps
\begin{equation}
\pi_i: Hol_{\mathbf k}^*(\bbp^1,LG^\infty/P_{LG}^\infty)\rightarrow SP^{k_i}(\bbc) = \bbc^{k_i}
\end{equation}
given by associating to a rational map $F$ the divisor of $F^* v_{\lambda_i}$, which sums to
\begin{equation}
\Pi: Hol_{\mathbf k}^*(\bbp^1,LG^\infty/P_{LG}^\infty)\rightarrow SP^{|{\mathbf k}|}(\bbc) = \bbc^{|{\mathbf k}|}
\end{equation}
given by associating to a rational map $F$ the divisor of $F^*(\prod_{\lambda\in \Lambda_P}v_{\lambda})$. Here ${|{\mathbf k}|}$ is the {\it scalar multiplicity} $\sum_{i=0}^{n(P)} k_i$. Here $SP$ denotes the symmetric product.

We can use the analytic structure of $Hol_{\mathbf k}^*(\bbp^1,LG^\infty/P_{LG}^\infty)$ to define a local principal parts space of multiplicity ${\mathbf k}$:
\begin{equation}{\mathcal {LP}}r_{\mathbf k} = (\pi_0,..,\pi_{n(P)})^{-1}(0,0,\dots ,0).
\end{equation}

Our theorem \ref{configurations} tells us in essence that over other points
\begin{equation}(\pi_0,..,\pi_{n(P)})^{-1}(a_0,\dots ,a_{n(P)})
\end{equation}
is a product of local principal part spaces ${\mathcal {LP}}r_{\mathbf m^i}$.  

\bigskip
\noindent{\it iii) Deformations and smoothness.} The infinitesimal version of this theorem is worth noting. Given a map $F$, corresponding to a bundle $E$ with reductions $R$ to $B_G$ at $\{\infty\}$ and a reduction to $N_G$ over  $\{0\}$, one wants to consider infinitesimal deformations of $F$ and of $E$. For the map,  a  deformation gives a unique section of $F^*(T(LG^a/P_{LG}^a))$; deformations of the bundle are given by $H^1(\bbp^1\times\bbp^1, ad(E)(0,-1))$; if one adds in the reductions at zero and infinity, one considers instead of $ad(E)$ the subsheaf $ad(E)(-R)$ of sections taking values in the $\gp_G$ subbundle at zero  and the $\gn_G$ subbundle at infinity corresponding to the reductions, so that deformations of the pair $(E,R)$ are given by $H^1(\bbp^1\times\bbp^1, ad(E)(-R)(0,-1))$. (The twist $(0,-1)$ represents the basing condition at $x=\infty$).

Now let us consider the sequence of sheaves \ref{ppt-sheaf}, and consider deformations of global sections, taking into account the basing conditions. One has
\begin{align}\label{ppt-sheaf-defn}
0\rightarrow H^0(\bbp^1, \gn_{LG}( -1))&\rightarrow H^0(\bbp^1, F^*(T(LG^a/P_{LG}^a))(-1)) \rightarrow H^0(\bbp^1, Q)\cr &\rightarrow H^1(\bbp^1, \gn_{LG}( -1))\dots \end{align}
Here $Q$ represents infinitesimal deformations of $F^*{\mathcal P}r$; using the definitions of ${\mathcal P}r$, this is identified in a straightforward fashion with $R^1\pi_*(ad(E)(-R)(0,-1))$, where $\pi$ is the projection
$\pi:\bbp^1\times\bbp^1\rightarrow\bbp^1)$. In turn, the groups $H^i(\bbp^1, \gn_{LG}(-1))$ are zero, as $\gn_{LG}$ has a filtration $\gn_i$ with abelian quotients, and $H^i(\bbp^1, {\mathcal O}(-1)) = 0$. This then gives
\begin{proposition}\label{isomorphism}
\begin{align*}H^0(\bbp^1, F^*(T(LG^a/P_{LG}^a))(-1))&= H^0(\bbp^1,R^1\pi_*(ad(E)(-R)(0,-1)))\cr &= H^1(\bbp^1\times \bbp^1, ad(E)(-R)(0,-1)).
\end{align*}
\end{proposition}

We have
\begin{proposition}
The space of maps $Hol_{\mathbf k}^*(\bbp^1,LG^\infty/P_{LG}^\infty)$ is smooth. 
\end{proposition}
It suffices to show that the space of deformations $H^1(\bbp^1\times \bbp^1, ad(E)(-R)(-1))$ of the equivalent data $(E,R)$ is of constant  rank; for this it suffices that $H^i(\bbp^1\times \bbp^1, ad(E)(-R)(0,-1)) = 0, i = 0,2$. To see  that $H^0(\bbp^1\times \bbp^1, ad(E)(-R)(0,-1)) = 0$, we use the fact that $ad(E)(-R)$ is trivial, and so $ad(E)(-R)(0,-1)$ is negative, over the open set of  lines $\{z\}\times \bbp^1$ over which $E$ is trivial. Dually, $H^2(\bbp^1\times \bbp^1, ad(E)(-R)(-1)) = H^0(\bbp^1\times \bbp^1,ad(E)(-R))^*(-1, -2)) = 0$ for the same reason.

As  $H^1(\bbp^1, ad(E)(-R)(-1))$ represents effective deformations of $(E,R)$, we also have that $H^0(\bbp^1, F^*(T(LG^a/P_{LG}^a))(-1))$ represents effective deformations of $F$. 

\bigskip
\noindent{\it iv) Dimensions and a formal degree for $F^*(T(LG/P_{LG} )$}. In finite dimensions, one can calculate the degree of the pullback under a map $F:\bbp^1\rightarrow G/P$ of the tangent bundle to a flag manifold, in terms of weights of the group. On the other hand, by Riemann-Roch this degree is equal to the Euler characteristic of $F^*(T(LG/P_{LG}))(-1)$. Elements of $H^0(\bbp^1,F^*(T(LG/P_{LG} ))(-1))$ correspond to the tangent vectors to the space of based maps to the flag manifold; on the other hand,  a positivity argument tells us that $H^1(\bbp^1, F^*T(G/P)(-1)) =0$, and so our Euler characteristic is equal to the dimension of the space of based maps, which is thus computed in terms of weights of the group. 

In our case, the analogous argument does not work, as the tangent bundle is infinite dimensional; nevertheless, one can compute the dimension of the corresponding space of (framed) bundles, which turns out to be given by the same formula in terms of weights. In this sense one has a formal degree for $F^*(T(LG/P_{LG} ))$.

We recapitulate the finite dimensional argument. The determinant of the bundle $F^*(T( G /P ))$ is  a homogeneous line bundle arising from a character $\epsilon$ of $P$ acting
on $\det(\gg / \gp)$.  As the fundamental 
weights span the dual to the torus space we certainly have that 
$\epsilon = \sum_{i \in \triangle} n_i\lambda_i$.
But if $\sigma_i$ is a simple root reflection in the Weyl group 
for a simple root in $\triangle_0$ then $\sigma_i$ is represented by 
an element of $P$.  However $\epsilon$ is a character of $P$ so 
$\sigma(\epsilon) = \epsilon$. But $\sigma_i(\lambda_j) = 
\lambda_j - \delta^i_j \alpha^j$ so  $n_i = 0$ whenever
$i$ is in $\triangle_0$ so that 
$$
\epsilon = - \sum_{i \not\in \triangle_0}n_i \lambda_i
$$
and  it follows that 
\begin{align}
d =& 2 c_1(\det(F^{-1}TG/P))\cr
       =& 2 \sum_{i \not\in \triangle_0}n_i k_i\cr
\end{align}
where $k_i$ are the degrees of the map into $G/P$. 

It is useful to have a formulae for the $n_i$ that we can compute with.  Note first that $\epsilon = 2(\rho - \rho_P)$
where $\rho$ is half the sum of the positive roots and the $\rho_P$ is half the sum of the positive roots in $P$. (This is where the finite dimensionality comes in.)
 Standard results (see for instance \cite{Humphreys1}) tell us that 
$$
\rho = \sum_{i \in \triangle} \lambda_i
$$
and that
$$
\rho_P = \sum_{i \in \triangle_0}
		\hat\lambda_i,
$$
where $\hat\lambda_i$ is the orthogonal projection, with the Killing form, 
of $\lambda_i$ onto  $\gt_P^*$. That is the $\hat \lambda_i$ are the 
fundamental weights for $\gg_P$ extended to be zero on $\gt_P^\perp$. We have 
for $i\in \Delta_0$:
$$
\hat \lambda_i = \lambda_i - \sum_{j \not\in \triangle_0}
			n_{ij}\lambda_j
$$ 
for some constants $n_{ij}$. For each $j \not\in \triangle_0$, let 
$$
N_j = \sum _{i \in \triangle_0} n_{ij}
$$ 
then 
$$
\epsilon = 2(\rho-\rho_P) = 2(\sum_{i \not\in \triangle_0}
		(\lambda_i + \sum_{j \in \triangle_0}
			n_{ji}\lambda_j),
$$
so
$$
d = 4  \sum_{i \not\in \triangle_0} (1+ N_i) k_i.
$$					
It is interesting to see that this formula can be applied also for loop groups because, although there are infinitely many positive roots, the weights $\rho$ and $\rho_P$ are
still well-defined as sums of fundamental weights.  We will check that this formula gives us the right dimension.

\begin{proposition}
The dimension of $H^0(\bbp^1,F^*(T(LG/P_{LG} ))(-1))$ is 
$$4  \sum_{i \not\in \triangle_0} (1+ N_i) k_i.$$
\end{proposition}
We begin with the maximal parabolic $P_{\max}$  in $LG$ corresponding to all of $G$, i.e. with the case 
$$
\triangle_0 = \{ \alpha_1, \dots, \alpha_r \}.
$$
We want the dimension of the space of framed bundles, which will be given by $h^1(\bbp^1\times \bbp^1, ad(E) (0,-1))$.  By our smoothness result, 
this is equal to minus the Euler characteristic. This, in turn, using  Atiyah's correspondence between instantons and spaces of rational maps
 into $LG/P_+{LG}$, was calculated in terms of instantons by Atiyah, Hitchin and Singer\cite{AHS}. We want the dimension of the space of 
framed instantons which is obtained from their formula by adding on the 
dimension of $K$. This gives  
$$
 4 \frac{1}{<\theta,\theta>}k,
$$
where $k$ is the charge of the instanton and $\theta$ is the longest root
of $K$.  

Now recall from section 2 that the simple roots of the affine Lie algebra
corresponding to the loop group of $K$ are obtained by adding 
$$
\alpha_0 = \delta - \theta
$$
to the roots $\alpha_1, \dots, a_r = \alpha_r$ of the finite dimensional algebra and the corresponding fundamental weights are 
$$
l_0 = \gamma, l_1 =\frac {2<\theta, \lambda_1>}{<\theta,\theta>}+\lambda_1, \dots,
l_r = \frac{2<\theta, \lambda_r>}{<\theta,\theta>}+\lambda_r.
$$
Here as above $\theta$ is the highest root and we have  modified the
formula in \cite{Mac} where it is  assumed that the Killing form has been
normalised so that $<\theta, \theta> = 2$. In our discussion $<\ ,\ >$
is precisely the Killing form.  To obtain the desired
$$\frac{1 }{<\theta,\theta>} = 1+N_0,$$
with $$
N_0 = \sum_{i=1}^{r}n_{i0},\quad \hat l_i = l_i - n_{i0} \gamma
$$
and 
$$
 n_{i0} = \frac{2<\theta, \lambda_i>}{<\theta,\theta>},
$$
so that
$$
N_0 = \sum_{i=1}^r \frac{2<\theta, \lambda_i>}{<\theta,\theta>}
=  \frac{2<\theta, \rho>}{<\theta, \theta>},
$$
we need that
$$
\frac{2<\theta, \rho>}{ <\theta, \theta>}+ 1 = \frac{1 }{ <\theta, \theta>}.
$$
This can be checked case by case from the formulae in the appendix to Bourbaki \cite{B}. 
A proof can also be given as follows. It is, of course, enough to know that
\begin{equation}
2<\rho,\theta> + <\theta,\theta> = 1. \label{toproot}
\end{equation}
Recall
that the Casimir operator for $G$ in the adjoint representations, $c_{\rm ad}$,
commutes with the adjoint action of $K$ so, by Schur's lemma
is a scalar multiple of the identity. This scalar multiple is
$$ 
\langle \rho + \theta , \rho + \theta \rangle - \langle \rho, \rho \rangle
$$
or  just  the LHS of \ref{toproot}.  So we have
$$
\tr(c_{\ad}) = \dim(K) ( 2\langle \rho, \theta\rangle + \langle \theta, \theta\rangle).
$$

On the other hand  the Casimir  is defined by 
$$
c_{\ad} = \sum_{i=1}^{\dim(K)} \ad(X_i)\ad(Y_i),
$$
where $X_i$ is a basis for $\gk$ and $Y_i$ a dual basis with respect 
to the Killing form. So 
\begin{align}
\tr(c_{\ad})  &= \sum \tr(\ad X_i \ad Y_i)    \cr
			&= \sum\langle X_i , Y_i \rangle     \cr
			&= \dim(K) \cr 
			&= \dim(K)(2<\rho,\theta> + <\theta,\theta>), \cr
\end{align}
which is what we wanted to prove.

To extend to the more general case of when $\Delta_0$ is a subset of $\{ \alpha_1, \dots, \alpha_r \}$, so that $P_{LG}$ corresponds to a parabolic $Q$ of $G$, we simply need to use the finite dimensional result, and the exact sequence of tangent bundles
$$
0\rightarrow T(G/Q)\rightarrow T(LG/P_{LG})\rightarrow T(LG/P_{\max})\rightarrow 0$$
This gives again for the formal degree (Euler characterisitic)
$$
d = 4  \sum_{i \not\in \triangle_0} (1+ N_i) k_i.
$$	

Finally, to deal with the case of an exotic parabolic, we can move down from the full flag manifold, and use the exact sequence 
$$
0\rightarrow T(B_{LG}/P_{LG})\rightarrow T(LG/B_{LG})\rightarrow T(LG/P_{LG})\rightarrow 0$$
and use the result for $T(LG/B_{LG})$ and that for the finite dimensional manifold, $B_{LG}/P_{LG}$ which is the full flag manifold for the Levi factor of $P_{LG}$.

\bigskip
\noindent{\it v) Deformations and multiplicity.}
The complement $Z$ of the big cell in $LG/P_{LG}$ is cut out by the sections $v_\lambda$ of $L_\lambda$. As we deform the principal part, the pull -back of $v_\lambda$ moves. Infinitesimally, given a tangent vector $X$ at a point of $Z_\lambda$, we can take a derivative $X(v_\lambda)$ of the defining section $v\lambda$; this is realised explicitly by
\begin{align}
T(LG/P_{LG})= LG\times_{P_{LG}}L\gg/\gp_{LG}&\rightarrow L_\lambda/Im (v_\lambda)= LG\otimes_{P_{LG}}\bbc_\lambda/Im (v_\lambda)\cr
(g,a)&\mapsto (g,\pi(-ag^{-1}v_\lambda)).
\end{align}
One checks that this is well defined and that it maps the image of $\gn_{LG}$ in $T(LG/P_{LG})$ to zero, so that one has a map $T(LG/P_{LG})/\gn_{LG}\rightarrow L_\lambda/Im (v_\lambda)$. Pulling this back via $F$, one then has the deformation map
\begin{equation}
R^1\pi_*(ad(E)(-R))\rightarrow F^*(L_\lambda/Im (v_\lambda)).
\end{equation}

\section{Configurations of principal parts, and stabilisation}

The preceeding description of Theorem \ref{configurations} of $G$-bundles on $\bbp^1\times\bbp^1$ (equipped with flags) as sections over $\bbp^1$ of a sheaf of principal parts gives us in some sense an algorithmic way of constructing them: if one wants to build a bundle (with flag) of total degree ${\mathbf k} = (k_0,k_1,\dots ,k_{n(P)})$, one first chooses $r$ multiplicities ${\mathbf m}^i = (m^i_{0}, m^i_{1},. \dots ,m^i_{n(P)}), i = 1,\dots ,r$ summing to $(k_0,k_1,\dots ,k_{n(P)})$ and $r$ points $p_i$ in $\bbc\subset\bbp^1$; one then glues in to the trivialized bundle near each $p_i$ a principal part in  ${\mathcal {LP}}r_{{\mathbf m}^i}$. This describes the moduli space $Hol_{\mathbf k}^*(\bbp_1, LG/P_{LG})$ of based maps as a space of labelled particles over $\bbc$, where the labels are the multiplicities and the principal parts; the space is stratified by the pattern of multiplicities of the particles.

Let $\delta_i$ be the multiplicity $(0, \dots 0, 1, 0, \dots ,0)$, with $1$ in the $i$-th position. This description of $Hol_{\mathbf k}^*(\bbp_1, LG/P_{LG})$  allows the definition of a stabilisation map
\begin{equation}
T_i: Hol_{\mathbf k}^*(\bbp_1, LG/P_{LG})\rightarrow Hol_{\mathbf k + \delta_i}^*(\bbp_1, LG/P_{LG});
\end{equation}
one simply adds  to a configuration of principal parts supported in a disk of radius $r$ around the origin, with one of them located on the circle of radius $r$, a principal part of multiplicity $\delta_i$ at the point $r+1$ on the real axis. This stabilisation preserves the stratification by multiplicity patterns.

We want to understand the behaviour of the homology of $Hol_{\mathbf k}^*(\bbp_1, LG/P_{LG})$ as ${\mathbf k}$ increases.  The homology of spaces of labelled particles has been extensively studied (see, e.g, the book of Cohen, Lada and May \cite{CLM}).  The sources of cycles, are, not surprisingly, two-fold: either cycles in the space of labels, or cycles in configurations of points. Roughly, on each multiplicity stratum, the homology stabilises as one adds in more points of multiplicity $\delta_i$; a more precise statement will follow. On the whole space, the homology will stabilise also, provided that the codimension of the strata increases with the number of points of high multiplicity in the stratum; again we will give a more precise statement below.

Thus, as we shall see,  the  homology of the space   stabilises; the question will then be to what. We will  appeal to a result of Gravesen \cite{Gr}, which will tell us that the limit is homology equivalent to the space of continuous maps $Map_{\mathbf k}^*(\bbp_1, LG/P_{LG})$; this space, up to homotopy, is independent of $k$. 
We note that from a variational point of view, this result is quite appealing. The space $Hol_{\mathbf k}^*(\bbp_1, LG/P_{LG})$, when $P_{LG}= LG^+$, represents the moduli space of instantons on the four sphere, while the corresponding space of continuous maps is the space of all connections on the four-sphere. The instantons minimise the energy functional on connections; thus our statement is that as one increases charge, the homology of the space of minima approximates that of the whole space. This result, which is known as the Atiyah-Jones conjecture, is already established for the classical groups; we now give a proof for all compact groups. Another case of stabilisation covered by the result is that of calorons, instantons on $\bbr^3\times S^1$ satisfying appropriate boundary conditions. 

We also show that the result can be extended from homology groups to homotopy groups. 

The results are an almost verbatim application of \cite{BHMM1, BHMM2, BHM}; we will mostly be quoting theorems, and only outlining the points which need modification.

\bigskip
\noindent{\it i) Stratification and codimensions}

We now describe in more detail the stratification of our space $Hol_{\mathbf k}^*(\bbp_1, LG/P_{LG})$. Following \cite{BHMM1}, we define a stratum associated to each collection of multi-indices
\begin{equation}
{\mathcal M} = \{{\mathbf m}^1,\dots ,{\mathbf m}^r\}
\end{equation}
such that $\sum_{i=1}^r{\mathbf m}^i = {\mathbf k}$. The stratum consist of configurations with $r$ ``poles", of multiplicities $\{{\mathbf m}^1,\dots ,{\mathbf m}^r\}$. One can project from a principal part to its location;  associated to ${\mathcal M} $, we have a smooth stratum ${\mathcal C}_{\mathcal M}$ of $SP(\bbc^{k_0})\times SP(\bbc^{k_1})\times\dots \times SP(\bbc^{k_{n(P)}})$ consisting of configurations of $r$ points with multiplicities ${{\mathbf m}^1,\dots ,{\mathbf m}^r}$, and a projection:  
\begin{equation}
(\pi_0,..,\pi_{n(P)}): {\mathcal S}_{\mathcal M} \rightarrow{\mathcal C}_{\mathcal M}
\end{equation}
The fiber over a point of ${\mathcal C}_{\mathcal M}$ is the product ${\mathcal LP}r_{{\mathbf m}^1}\times {\mathcal LP}r_{{\mathbf m}^2}\times \dots \times {\mathcal LP}r_{{\mathbf m}^r}$. We note that the strata ${\mathcal S}_{\mathcal M}$ are not necessarily smooth. We can, however, stratify it by smooth varieties, for example using the Whitney stratification
\begin{equation}
{\mathcal S}_{\mathcal M} = \cup{\mathcal S}_{{\mathcal M},j}.
\end{equation}
Just as in \cite {BHMM1}, section 6, one can then show that one can order this finite set of strata in such a way that
\begin{itemize}
\item {}The lowest order element is a dense open set; more generally, the unions of all submanifolds of a given order  $\ell$ or less is dense and open
\item {}The normal bundle of the $\ell$-th stratum is orientable in the union of strata of order less than $\ell$.
\end{itemize}

We now want to estimate the codimension of the strata ${\mathcal S}_{\mathcal M}$ in terms of the multiplicity. We follow, with some variation, the procedure of \cite{BHMM1}. There were, in \cite{BHMM1}, three steps;

\begin{itemize}
\item {}One shows, for any $\ell$ and any  principal part $R$ of degree ${\mathbf m}$, that there is a map $F$ of some degree ${\mathbf K}$ with $R$ as one of its principal parts, at $p$, say, such that  the family of holomorphic maps around  $F$   submerges at $p$ onto the jets of a given order $\ell$ of maps into $G/P$;
\item {} One then bounds at a pole the codimension of the variety of jets of multiplicity $\mathbf m$ within the variety of jets. 
\item{} The two first steps bound the codimension within the variety of maps $Hol_{\mathbf K}$; one then uses the local decomposition $Hol_{\mathbf K} = Hol_{\mathbf m} \times Hol_{\mathbf K-m}$ induced by the principal parts picture  to get the same codimension within $Hol_{\mathbf m}$.
\end{itemize}

In our case, as we are in infinite dimensions, the first step must be modified. Fortunately, we have  theorem \ref{birkhoff-strata}, which tells us that at a point of the strata at infinity (which point we will take to be $wP_{LG}$, without loss of generality), there is a decomposition of the space into an infinite dimensional orbit of $U_{LG}^w = N^-_{LG}\cap wU^-_{LG}w^{-1}$, which preserves the stratification of $LG/P_{LG}$, and, transversally, a finite dimensional orbit of $A_w = N_{LG} \cap wU^-_{LG}w^{-1}$. It suffices to show, then, that there is a submersion, for high degree, onto the ``jets in the $A_w$-direction". This amounts to asking that the quotient bundle
\begin{equation}
0\rightarrow \gu_{LG}\cap w\gn_{LG}w^{-1}\rightarrow F^*T(LG/P_{LG})\rightarrow Q\rightarrow 0
\end{equation}
be a sum $\oplus{\mathcal O}(a_i)$ with all $a_i$ at least $\ell$. This means that the map on global sections can have arbitrary jets at a point, and so the application will be submersive.

Let us therefore consider an element $R$ of ${\mathcal {LP}}_r$ of degree ${\mathbf m}= (m_0,\dots ,m_{n(P)})$ evaluating to $wP_{LG}$. We choose a map $\sigma:\bbp^1\rightarrow \bbp^1$ of degree $\ell$, with $\sigma^{-1}(1) = \{p_1,\dots ,p_{\ell}\}$ distinct points. We first create a map $F$ in  $Hol_{\mathbf m}^*(\bbp_1, LG/P_{LG})$ with a single pole by glueing in at $z=1$ the principal part $R\circ\sigma^{-1}$, where $\sigma^{-1}$ is the local inverse of $\sigma$ at $p_1$.
The composition $F\circ\sigma$ then has principal part $R$ at $p_1$, as well as   $\ell-1$ other principal parts of multiplicity ${\mathbf m}$ at $\{p_2,\dots ,p_\ell\}$. It also has $\ell$ points $\{q_1,\dots ,q_{\ell}\}$ counted with multiplicity (the points in  $\sigma^{-1}(\infty)$)  at which the map takes value the base point $P_{LG}$. Now let us look at the image in $Q$ of the trivial bundle $\gu_{LG}^+\cap w\gn_{LG}w^{-1}$; it spans $Q$ at $p_1$ (and so its image is of full rank), and its image vanishes at $\ell$ points $q_i$. This implies that $Q$ is indeed a sum $\oplus{\mathcal O}(a_i), a_i\ge \ell$. Thus, transversally to the orbit of $U_{LG}^w = N^-_{LG}\cap wU^-_{LG}w^{-1}$, the map to jets is locally surjective.

We now look at the space of $J^{\ell}$ of $\ell$ jets into $LG/P_{LG}$, and ask ourselves what is the codimension in this set of the multiplicity ${\mathbf m}$ jets. If the divisor "at infinity" $Z$, whose complement is the big cell, were smooth with normal crossings, one would have codimension $|{\mathbf m}|$, as is easily seen from a computation in local coordinates. Let $v=0$ be a defining equation for $Z$. If the divisor is not smooth, one still has \cite{BHMM1}, by blowing up:

\begin{proposition}\cite{BHMM1}, 
 There are constants $c>0$ and 
$c^{\prime}\ge 0$   independent of $|{\mathbf m}|$ such that 
the codimension of the subvariety of the  jets such that the pull back of $v$ vanishes to order $|{\mathbf m}|$ is 
at least $ {c \cdot |{\mathbf m}|- c^{\prime} }$.
\end{proposition}

The submersion result then implies that  the same codimension holds at our point in the  the space of holomorphic maps of sufficiently high degree ${\mathbf K}$, by our submersion result. On the other hand, near our map, one has a local decomposition $Hol_{\mathbf K} = Hol_{\mathbf m}\times Hol_{{\mathbf K}-{\mathbf m}}$, and so the codimension holds inside of $Hol_{\mathbf m}$. Thus, the codimension in $Hol_{\mathbf m}$ of the maps with one simple pole of multiplicity $\mathbf m$ is bounded below by $c|{\mathbf m}|-c'$.

Doing this for each pole of multiplicity ${\mathbf m}^i$ in the stratum gives:

\begin{proposition}(\cite{BHMM1}, Proposition 6.6)
There exists a
positive constant $c(LG/P_{LG})$ which is independent of the stratum
indices ${\mathcal M}$ and the multi-degree ${\mathbf k},$ so that the complex
codimension of $S_{{\mathcal M}}$ in  $Hol_{\mathbf k}^*(\bbp_1, LG/P_{LG})$
is bounded below by 
$$c(LG/P_{LG})\sum_i (|{\mathbf m}^i|-1).$$  
\end{proposition}

\bigskip
\noindent{\it ii) Stabilisation on strata, and on the whole space.}
We have defined stabilisation maps $T_i$ which to a configuration of principal parts add a fixed principal part of multiplicity $\delta_i$. For a stratum ${\mathcal S}_{\mathcal M}$, let $s_i({\mathcal M})$ be the number of multiplicities in the collection ${\mathcal M}$ equal to $\delta_i$. One has the general theorem from \cite{BHMM1},\cite{BHMM2}:
\begin{theorem}
For all coefficient rings, the stabilisation map 
\begin{equation}
T_i: {\mathcal S}_{\mathcal M}\rightarrow  {\mathcal S}_{{\mathcal M}\cup \{\delta_i\}}
\end{equation}
induces isomorphisms in homology groups $H_j$ for $j\le  [s_i({\mathcal M})/2]$.
\end{theorem}
Roughly speaking, the cycles in ${\mathcal S}_{{\mathcal M}\cup \{\delta_i\}}$
for which one needs the extra labelled point only occur in dimensions greater than $ [s_i({\mathcal M})/2]$. 

The homology groups of the strata fit together through a Leray spectral sequence to give the homology of the whole space. Let $I_k$ be the index set for the strata.

For each ${\mathbf k}$ and all coefficient rings $A,$ there are homology Leray spectral
sequences $E^r(Hol_{\mathbf k}^*(\bbp_1, LG/P_{LG});A)$ which converge to filtrations of the homology
$H_*(Hol_{\mathbf k}^*(\bbp_1, LG/P_{LG});A)$ with 
$$ E^1(Hol_{\mathbf k}^*(\bbp_1, LG/P_{LG});A) ~\cong~ \bigoplus_{{\mathcal I}_{\mathbf k}} 
H_*(\Sigma^{2cd({\mathcal M},j,{\mathbf k})}({\mathcal S}_{{\mathcal M},j})_+;A)$$
where $2cd({\mathcal M},j,{\mathbf k})$ is the real codimension of 
${\mathcal S}_{{\mathcal M},j}$ in $Hol_{\mathbf k}^*(\bbp_1, LG/P_{LG}).$ The homology of the strata thus appears in the spectral sequence, suspended by the codimension. Furthermore, the inclusions $T_i$ induce maps of spectral sequences
$$\tilde T_i: 
E^r(Hol_{\mathbf k}^*(\bbp_1, LG/P_{LG});A) \longrightarrow E^r (Hol_{{\mathbf k}+\delta_i}^*(\bbp_1, LG/P_{LG});A).$$

Building first  the homology of the strata, we see first that the homology of the strata with many poles of order one stabilises; then, in the spectral sequence the homology of strata with few poles of order one gets suspended upwards by a high codimension. Combining the two effects gives us a stabilisation for homology. 

If ${\mathbf k},{\mathbf k}'$ are two multi-indices, we set ${\mathbf k}<{\mathbf k}'$ if 
${\mathbf k}'-{\mathbf k}$ contains only zero or positive entries. Given two such multi-indices, one can define a stabilisation
$T_{{\mathbf k},{\mathbf k}'}:Hol_{\mathbf k}^*(\bbp_1, LG/P_{LG})\rightarrow Hol_{\mathbf k'}^*(\bbp_1, LG/P_{LG})$ by composing the appropriate number of $T_i$'s. (These maps commute up to homotopy).
Now set $\ell({\mathbf k}) = \min(k_i)_{i=1,\dots ,n(P)}$.

\begin{theorem}\cite{BHMM1} For all coefficient rings $A$,
the stabilisation 
$$T_{{\mathbf k},{\mathbf k}'}:Hol_{\mathbf k}^*(\bbp_1, LG/P_{LG})\rightarrow Hol_{\mathbf k'}^*(\bbp_1, LG/P_{LG})$$ 
induces isomorphisms in homology groups $H_j$ for
$j\le [\min(1/2, c(LG/P_{LG}))\ell({\mathbf k})] - 1$.
\end{theorem}

\bigskip
\noindent{\it iii) Stabilisation to the space of continuous maps}

We have now seen that the homology of our spaces stabilises; as noted above,  the question is then to what. For this, we can exploit the result of Gravesen \cite{Gr}, who builds a limit space of configurations of principal parts, and shows that it is homology equivalent to the space of based continuous maps $\Omega^2(LG/P_{LG})$. One then has, as in \cite{BHMM1}, with the proofs going over verbatim:

\begin{theorem}\cite{BHMM1} For all coefficient rings $A$,
the inclusion 
$$I:Hol_{\mathbf k}^*(\bbp_1, LG/P_{LG})\rightarrow \Omega^2_{\mathbf k}(LG/P_{LG})$$ 
induces isomorphisms in homology groups $H_j$ for
$j\le [\min(1/2, c(LG/P_{LG}))\ell({\mathbf k})] - 1$.
\end{theorem}
 
\bigskip

\noindent{\it iv) Homotopy groups}

There remains the question of extending the results to homotopy groups. For this, one shows that the homology of the universal covers stabilises, and then appeals to Whitehead's theorem.
The proof, once one has the principal parts description, repeats the analysis of \cite{BHM}, and we will not go through it here, except to remark that one shows first that the fundamental group is Abelian, as soon as one has $k_i>1$ for each $i$. The reason for this is that one can arrange for a loop to live in a configuration with one multiplicity one pole of each type. If one has two loops in the space, and all $k_i$ at least two, one can place one loop in a configuration over a disk, and the second loop in a configuration over a second disjoint disk. The two loops then commute for obvious geometrical reasons. Once one has this fact, the rest of the analysis in \cite{BHM} involves understanding the homology of Abelian covers of particle spaces enough to prove stabilisation. This analysis is not straightforward, but yields:

\begin{theorem}\cite{BHMM1} The inclusion 
$$I:Hol_{\mathbf k}^*(\bbp_1, LG/P_{LG})\rightarrow \Omega_{\mathbf k}^2(LG/P_{LG})$$ 
induces isomorphisms in homotopy groups $\pi_j$ for
$j\le [\min(1/2, c(LG/P_{LG}))\ell({\mathbf k})] -s - 2$, where $s$ is the rank of $\pi_3(LG/P_{LG})$.

\end{theorem}

In particular, the homotopy groups of the moduli spaces of $G$-bundles  $\bbp^1\times \bbp^1$  with zero degree, trivialised on $\{\bbp^1\times \infty\}\cup\{\infty\times \bbp^1\}$ and a reduction to a parabolic $P$ over $\{\bbp^1\times 0\}$ stabilise as one increases both the second Chern class and the degree of the flag.

\bigskip
\noindent{\it v) Atiyah-Jones type conjectures.} 
\bigskip

{\it v.1. Instantons.}  Let $P_{LG} = LG^+$, and let $G = K_\bbc$ for a compact semi-simple group $K$. One has the theorem, due to Donaldson \cite{Donaldson}:

\begin{theorem}
The moduli space ${\ca M}_{K,k}$ of framed $K$-instantons of charge $k$ on the four-sphere is isomorphic to the moduli space $Hol_{ k}^*(\bbp_1, LG/LG_+)$
\end{theorem}
Coupled with the isomorphism $\Omega^2(LG/LG_+)\simeq \Omega^3K$, one then has, as a result of our theorem:
\begin{corollary}(Atiyah-Jones conjecture)
There are positive constants $c, c'$ such that the  inclusion ${\ca M}_{K,k}\rightarrow \Omega^3K$ induces isomorphisms in homotopy groups $\pi_j$ for
$j\le c|k|-c'$.
\end{corollary}
The space $\Omega^3K$, as noted by Atiyah and Jones \cite{atiyah-jones} is the space of all connections on the four sphere, modulo gauge. The statement is thus that the space of minima of the energy functional captures the topology as one increases the charge. As noted in the introduction, this result is proven for the classical groups \cite{BHMM2, Kirwan-AJ,Tian1,Tian2}; this gives it for all groups. The techniques of \cite{HM-ruled} should also apply to extend the result from the four sphere to ruled surfaces.
\bigskip

{\it v.2. Calorons.} Instead of   instantons, one can consider calorons.  As in \cite{GM, Charbonneau-Hurtubise,Nye-Singer}, $K$-calorons are \emph{anti-self-dual} $K$-connections on $\bbr^3\times S^1$, with asymptotic behaviour that resembles that of $\bbr^3$-monopole. We view $\bbr^3\times S^1$ as the quotient of the standard Euclidean $\bbr^4$ by the time translation $(t,x)\mapsto (t+2\pi/\mu_0,x)$. Let $A$ be  a $K$-connection on $\bbr^3\times S^1$; we write it in coordinates over $\bbr^4$ as  
\begin{equation} 
A = \phi dt+\sum_{i=1,2,3}A_idx_i. 
\end{equation} 
We require that the $L^2$ norm of the curvature of $A$ be finite, and that in suitable gauges, the $A_i$ be $O(|x|^{-2})$, and that $\phi$ be conjugate to $({\mathbf \mu}  - {\mathbf k}/2|x|^{-1} + O(|x|^{-2}))$, for constant elements of the Lie algebra ${\mathbf \mu},{\mathbf k}$. The element ${\mathbf k}$ lies in a cone in a lattice, and represents a \emph {monopole charge}. There are  also have bounds on the derivatives of these fields, which tell us in essence that in a suitable way the connection extends to the $2$-sphere at infinity times $S^1$; furthermore, one can show that the extension is to a fixed connection, which involves fixing a trivialisation at infinity; there is thus a second invariant we can define, the relative second Chern class, which we represent by a (positive) integer $k_0$, the \emph {instanton charge}. If this latter charge is zero, the connection is in fact simply a monopole, lifted from $\bbr^3$. 

One has, for $K= SU(2)$, the theorem (\cite{Charbonneau-Hurtubise}) :

\begin{theorem}
The moduli space ${\ca M}_{SU(2),  k_0, k_1}$ of framed $K$-calorons of instanton charge $k_0$ and monopole charge $k_1$  is isomorphic to the moduli space of maps $Hol_{k_0,k_1}^*(\bbp_1, LSL(2,\bbc)/B_{LSL(2,\bbc)})$, which in turn is the space of rank $2$ bundles of second Chern class $k_0$ on $\bbp^1\times \bbp^1$, trivialised on $\{\bbp^1\times \infty\}\cup\{\infty\times \bbp^1\}$, with a subline bundle of degree $-k_1$ on $\{\bbp^1\times 0\}$.
\end{theorem}

This theorem is quite likely true, with the necessary modifications, for all other compact groups $K$, though the techniques of \cite{Charbonneau-Hurtubise} are not likely to adapt well.  If we accept this as given, we then have:

\begin{corollary}(Atiyah-Jones conjecture for calorons.)
There are positive constants $c, c'$ such that the  inclusion ${\ca M}_{K,k_0,  {\mathbf k}}\rightarrow \Omega_{k_0,{\mathbf k}}^2(LG/P_{LG})$ induces isomorphisms in homotopy groups $\pi_j$ for
$j\le c|{\mathbf k}|-c'$.
\end{corollary}

\section{Outer automorphisms and Hecke transforms}

The elements of our loop group flag manifolds, when these flag manifolds correspond to flag manifolds on the finite dimensional group, yield bundles on $\bbp^1$, equipped with flags $F^0$ along $z=0$, as well as trivialisations on a neighbourhood of $z=\infty$. This trivialisation allows us also to choose a flag $F^\infty$ along $z=\infty$. 

At least for the classical groups,   this allows us to define Hecke transformations ${\mathcal H}$, giving us other bundles and flags; from the loop group point of view,  they are given by the action of the outer automorphisms. These, as we have seen, are classified by symmetries of the extended Dynkin diagram corresponding to the loop group. 

We consider each family of classical groups in turn:

\subsection{  Sl(n,\bbc).} Here, the outer automorphisms group is cyclic of order $n$, and generated by a single transformation defined as follows. Let $T$ be an element of $LSL(n,\bbc)/B_{LSL(n,\bbc)}$; it determines a  bundle $E$ on $\bbp^1$ equipped with flags 
$E^0_1\subset...\subset E^0_{n-1}$ over $z=0$, and a trivialisation over $|z|>1$, so that one  can choose in a fixed way a flag $E^\infty_1\subset...\subset E^\infty_{n-1}$ over $z=\infty$. We remember by (\ref{sl(n)-sheaves}) that  one not only has one sheaf, but a whole nested sequence $E^{i,j} = (E^0)^{i,j}$; likewise, the flag structure at infinity gives us in parallel $(E^\infty)^{i,j}$. Now define the Hecke transform ${\mathcal H}(E)$ by 
\begin{eqnarray*}
{\mathcal H}(E)  &=& E\quad {\rm away \ from\ }z=0,\infty,\cr
&& (E^0)^{0,n-1} \quad{\rm near\ } z=0,\cr
&& (E^\infty)^{1,1} \quad {\rm near \ }z=\infty.\cr
\end{eqnarray*}
In short, one allows a pole in one component at infinity, and forces one component to vanish at zero; the result still has degree zero, and indeed has an overall volume form. On the level of transition functions $T$, as we saw above, the Hecke transform operates by conjugating by the matrix $A(z^{-1})$, where $A(z)$ is the matrix (\ref{outer-sl}). The n-th power of the Hecke transform, up to inner automorphicms, is the identity.

This transform can of course be applied to families, in particular families parame\-trised by a Riemann surface. We saw that each such family $E$ had associated to it an $n$-tuple of degrees $K(E)= (k_1,...,k_n)$. One checks that the Hecke transform permutes the degrees cyclically by   $K({\mathcal H}(E))= (k_n, k_1,...,k_{n-1})$.

One   has an action permuting the components of $Hol(\Sigma, LSL(n,\bbc)/B_{LSL(n,\bbc)})$ of different degrees 
amongst themselves. 
In particular, for $\Sigma = \bbp^1$, the moduli spaces correspond to moduli of calorons, and the Hecke transform permutes these. 

There are several  special cases. The first is that of degree $(k,k,..,k) $, which  corresponds (see \ref{charges}) to instanton bundles, trivial (and with the trivial flag) along the line $z=0$ (this is, in fact, a dense subset of the instanton moduli); here the Hecke transform acts by automorphisms. These automorphisms have natural interpretations in terms of the monad representation of the bundles: indeed, the holomorphic bundles of rank $n$ with $c_2 = k$ corresponding to $SU(n)$-instantons are encoded in   four-tuples  of matrices $A, B, C,D$ with $A,B$  $k\times k$, $C$ $k\times n$, and  $D$ $n\times k$. These matrices satisfy the constraint $[A,B]+CD = 0$, as well as some non-degeneracy constraints, and the correspondence is bijective, up to the action of $Gl(k,\bbc)$ by $g(A,B,C,D) = (gAg^{-1}, gBg^{-1}, gC, Dg^{-1})$. The additional constraint of triviality along $z=0$ tells us that $A$ is invertible. Writing the columns of $C$ as $C_i$, and the rows of $D$ as $D_i$, $i=1,...,n$, the Hecke transform acts by 
$$(A,B,C_1,..,C_n, D_1,..,D_n)\mapsto (A, B-C_1D_1A^{-1}, C_2,..,C_n,A C_1, D_2,..,D_n,D_1A^{-1}).$$
If one iterates this $n$ times, this gives $(A, (BA-CD)A^{-1}= ABA^{-1}, AC, DA^{-1}) = A(A,B,C,D)$, bringing us back to the initial point.

Another  special case is  obtained by setting  the last degree to zero: one has  degree $(k_1,..,k_{n-1}, 0)$; these are trivial bundles with  flags of subbundles, and correspond  to monopoles. The Hecke transform permutes  these  spaces amongst themselves or with caloron spaces .

Similarly, replacing Borels by parabolics, one has actions of the Hecke transform mapping the moduli space for one parabolic to the moduli space for another.

These Hecke transforms on calorons have natural realisations in terms of Nahm transfoms of the calorons,  or monopoles. We recall that the Nahm transform attaches to each element of the moduli space of calorons of degree $(k_1,..,k_n)$ a solution to Nahm's equations on the circle (\cite{Nye-Singer}). The solution involves partitioning the circle into $n$ consecutive intervals $(\mu_i,\mu_{i+1})$ with $\mu_n$ and $\mu_1$ representing the same point. The $\mu_i$ represent appropriate physical constants associated to the problem (``masses"). On each interval, one has solutions to some non-linear o.d.e., Nahm's equations, given by $k_i\times k_i$ matrices, and at the boundary points, some boundary conditions. The Hecke transform simply amounts in this case to rotating the order of the intervals. We note that this also permutes  the masses, and could conceivably correspond to some form of physical duality, or rather, $n$-ality.

\appendix



\begin{thebibliography}{BHMM2}

\bibitem[A]{Atiyah} {M.F. Atiyah}, {\em Instantons in two and four
dimensions}, Comm. Math. Phys., 93 (1984), 437-451.
\bibitem[AJ]{atiyah-jones} {M.F. Atiyah, J.D. Jones}, {\em Topological
aspects of Yang-Mills theory}, Comm. Math. Phys. 61 (1978), 
97-118. 
\bibitem[AHS]{AHS} M.F. Atiyah, N.J. Hitchin and   I. Singer, {\em Self-duality in four dimensional Riemannian geometry}, Proc. Roy. Soc. London (Ser. A) 362 (1978), 425-461.
\bibitem[B]{B} N. Bourbaki, {\em Groupes et Alg\'ebres de Lie}, Paris, Hermann, 1968.
\bibitem[BE]{BasEas}{R. J. Baston and M. G. Eastwood}, {\sl The Penrose Transform: Its Interaction with Representation Theory}, 
Oxford Mathematical Monographs, Oxford University Press, 1989.
\bibitem[BHM]{BHM}{C.P. Boyer, J.C. Hurtubise,  R.J. 
Milgram}, {\em Stability theorems for spaces of rational maps}, Int. J. of Math. 12 (2001) 223-262
\bibitem[BHMM1]{BHMM1}{C.P. Boyer, J.C. Hurtubise, B.M. Mann, R.J. 
Milgram}, {\em The topology of holomorphic maps into generalized flag
manifolds}, Acta Math. 173 (1994) 61-101.
\bibitem[BHMM2]{BHMM2}{C.P. Boyer, J.C. Hurtubise, B.M. Mann, R.J. 
Milgram}, {\em The topology of instanton moduli spaces. I: The
Atiyah-Jones conjecture}, Ann. of Math. 137 (1993) 561-609.
\bibitem[CH]{Charbonneau-Hurtubise}
B.~Charbonneau and J.~Hurtubise.
 {\em Calorons, {N}ahm's equations on {$S^1$} and bundles over
  {$\mathbb{P}^1\times \mathbb{P}^1$},} Comm. Math. Phys. 280 (2008), 315--349. 
\bibitem [CLM]{CLM}{ F.R. Cohen, T.J. Lada, J.P. May}, {\em The Topology of 
Iterated Loop Spaces}, Springer Lecture Notes 533, Berlin, 1976.
\bibitem[Do]{Donaldson} {S.K. Donaldson}, {\em Instantons and geometric invariant theory}, Comm. Math. Phys. 93 (1984), no. 4, 453--460. 
\bibitem[GM]{GM}H.~Garland and M.~K. Murray.
{\em Kac-{M}oody monopoles and periodic instantons.}
  Comm. Math. Phys. 120(2) (1988):335--351.
\bibitem[Gra]{Gr} {J. Gravesen}, {\em On the topology of spaces of
holomorphic maps}, Acta Math. 162 (1989), 247-286.
\bibitem[Gu2]{Gu} {M.A.  Guest}, {\em The topology of the space of rational
curves on a toric variety}, Acta Math. 174 (1995) 119-145. 
\bibitem[Hum1]{Humphreys1} {J. E. Humphreys} {\em Introduction to Lie Algebras and Representation
Theory}, Springer-Verlag (1972).
\bibitem[Hum2]{Hum} {J. E. Humphreys}, {\em Reflection Groups and Coxeter Groups}. Cambridge studies in advanced
mathematics {\bf 29}, Cambridge University Press, 1990.
\bibitem[HM]{HM-ruled} {J.C. Hurtubise, R.J. Milgram}, {\em The Atiyah-Jones
conjecture for ruled surfaces}, J. fur die Reine und Ang. Math. 466, (1995),
111-143.
\bibitem[HuM]{HuM2} {J.C. Hurtubise, M.K. Murray}, {\em Monopoles and their spectral data}, Comm. Math. Phys.
133, 1990
pp.487-508 
\bibitem[Hu1]{Riemann} {J.C.  Hurtubise}, {\em Holomorphic maps of a Riemann
surface into a flag manifold}, J. Diff. Geom. 43, (1996), 99-118.
\bibitem[Ki1]{Kirwan-maps} {F.C. Kirwan}, {\em On spaces of maps from Riemann
surfaces to Grassmannians and applications to the cohomology of moduli of
vector bundles}, Ark. Math. 24(2) (1986), 221-275.  
\bibitem[Ki2]{Kirwan-AJ} {F.C. Kirwan}, {\em Geometric invariant theory and the
Atiyah-Jones conjecture}, Sophus Lie Memorial Conference Proceedings, O.A.
Laudal and B. Jahren eds., Scandinavian University Press, Oslo, 1994.
\bibitem[M]{Mac} {I.~G.~ Macdonald}, {\em Affine Lie algebras and modular forms}, in S\'eminaire Bourbaki Exp. {\bf 77}, Lecture
Notes in Mathematics Vol. {\bf 901}, 258--265. Springer-Verlag, New York, 1981.
\bibitem[N]{Nye} T.~M.~W. Nye, PhD thesis, {\em Geometry of Calorons}, {\tt hep-th/0311215}.
\bibitem[NS]{Nye-Singer} T.~M.~W. Nye and M.~A. Singer, {\em An {$L\sp 2$}-index theorem for {D}irac operators on {$S\sp 1\times\mathbf{R}\sp 3$}.} J. Funct. Anal. 177(1):203--218, 2000, arXiv:math.DG/0009144.
\bibitem[PS]{PS} {A. Pressley and G. Segal}, {\em Loop Groups}, Oxford Mathematical Monographs, Oxford University Press, 1986.

\bibitem[Se]{Segal} {G. Segal}, {\em The topology of rational functions},
Acta Math. 143 (1979), 39-72.  
\bibitem [T]{Taubes} {C.H. Taubes}, {\em The stable topology of
self-dual moduli spaces}, J.  Diff.  Geom., 29 (1989), 163-230.
\bibitem[Ti1]{Tian1} {Y.  Tian}, {\em The based
$\scriptstyle{SU(n)}$-instanton moduli spaces}, Math.  Ann. 298 (1994),
117-139.  
\bibitem[Ti2]{Tian2} {Y. Tian}, {\em The Atiyah-Jones conjecture for the
classical Lie groups and Bott periodicity}, J. Diff. Geom. 44 (1996), 178-199.


\end{thebibliography}
\end{document}